\documentclass[12pt,a4paper]{amsart}
\usepackage{graphicx}
\usepackage{amsmath,amsfonts}
\usepackage{amsthm,amssymb,latexsym}
\usepackage[active]{srcltx}
\usepackage{natbib}

\newtheorem{theorem}{Theorem}[section]
\newtheorem{corollary}[theorem]{Corollary}
\newtheorem{lemma}[theorem]{Lemma}
\newtheorem{fact}[theorem]{Fact}
\newtheorem*{fact*}{Fact}
\newtheorem{proposition}[theorem]{Proposition}
\theoremstyle{definition}

\numberwithin{equation}{section}

\newcommand{\N}{\mathbb N}
\newcommand{\Z}{\mathbb Z}

\newcommand{\F}{\mathbb F}
\newcommand{\Q}{\mathbb Q}
\newcommand{\K}{K}

\newcommand{\Pp}{\mathbb P}
\newcommand{\oK}{\hskip0.8mm\overline{\hskip-0.8mm\K\hskip-0.3mm}\hskip0.3mm}
\newcommand{\PpoK}{\Pp_{\!\!\oK}}
\newcommand{\PpF}{\Pp_{\!\mathbb{F}}}
\newcommand{\PpoKDd}{\Pp_{\!\!\oK}^{\ }{}^{\!\!\!\!\bfs D(\bfs d)}}

\newcommand{\bfs}{\boldsymbol}

\newcommand{\fq}{\F_{\hskip-0.7mm q}}

\def\ifm#1#2{\relax \ifmmode#1\else#2\fi}

\newcommand{\klk}    {\ifm {,\ldots,} {$,\ldots,$}}

\newcommand{\plp}    {\ifm {+\cdots+} {$+\ldots+$}}

\newcommand{\twodots}    {\, .. \,}

\begin{document}

\title[Irreducible smooth complete intersections]{Explicit estimates for
polynomial\\systems defining irreducible smooth\\complete
intersections}%
\author[von zur Gathen \& Matera]{
Joachim von zur Gathen${}^{1}$,
Guillermo Matera${}^{2,3}$,
}

\address{${}^{1}$B-IT\\
  Universit\"{a}t Bonn\\
  D - 53113 Bonn}
  \email{gathen@bit.uni-bonn.de}

\address{${}^{2}$Instituto del Desarrollo Humano,\\
Universidad Nacional de Gene\-ral Sarmiento, J.M. Guti\'errez 1150
(B1613GSX) Los Polvorines, Buenos Aires, Argentina}
\email{gmatera@ungs.edu.ar}
\address{${}^{3}$ National Council of Science and Technology (CONICET),
Ar\-gentina}

\thanks{JvzG acknowledges the support of the B-IT Foundation and the Land
Nordrhein-Westfalen. GM is partially supported by the grants UNGS
30/3084, PIP CONICET 11220130100598 and PIO CONICET-UNGS 14420140100027.}%
%
%

\date{\today}%
\begin{abstract}
This paper deals with properties of the algebraic variety defined as
the set of zeros of a ``typical'' sequence of polynomials. We
consider various types of ``nice'' varieties: set-theoretic and
ideal-theoretic complete intersections, absolutely irreducible ones,
and nonsingular ones. For these types, we present a nonzero
``obstruction'' polynomial of explicitly bounded degree in the coefficients of
the sequence that vanishes if its variety is not of the type. Over
finite fields, this yields bounds on the number of such sequences.
We also show that most sequences (of at least two polynomials)
 define a degenerate variety,
namely an absolutely irreducible
nonsingular hypersurface in some linear projective subspace.
\end{abstract}

\maketitle
%
%
\section{Introduction}
Over a field $K$,  a sequence $\bfs f = (f_1\klk f_s)$ of
homogeneous polynomials in $n+1$ variables
with $n>s$ defines a projective
variety $V \subseteq \Pp_{\!\!K}^n$, namely, its set of common roots.
Intuitively, most such sequences are regular and $V$ enjoys ``nice''
properties, such as being a set-theoretic or ideal-theoretic
complete intersection, being (absolutely) irreducible, and
nonsingular. This paper confirms this intuition in a quantitative
way.

For a fixed pattern $(d_1,\ldots,d_s)$ of degrees $d_i = \deg f_i$,
the set of all such $\bfs f$ forms a multiprojective space in a
natural fashion. For properties as above, we provide a nonzero
``obstruction polynomial'' $P$ of explicitly bounded degree in
variables corresponding to the coefficients in $\bfs f$ such that
any $\bfs f$ with $P(\bfs f) \neq 0$ enjoys the property. Thus
``most'' sequences define a nice variety.

If $K$ is finite with $q$ elements, we obtain as a consequence
bounds on the probability that the variety is nice.
They have the form $1-O(q^{-1})$ with explicit constants
 depending on the geometric data, but not on $q$.

For each property, we
first present an obstruction polynomial as above that works for any
field. From this, we derive numerical estimates in the case of
finite fields.

Section \ref{section:notations} provides some notational background.
In Sections \ref{section:complete} and \ref{section:irreducible} we
fix the degree sequence of our polynomial sequence and study four
geometric properties of the corresponding projective variety in the
appropriate projective space: being a set-theoretic or an
ideal-theoretic complete intersection, absolute irreducibility and
nonsingularity. For these properties, we present a nonzero
``obstruction'' polynomial of bounded degree in variables
corresponding to the coefficients of the polynomial sequence that
vanishes if the corresponding variety is not of the type; see the
``geometric'' Theorems \ref{th: V is set-theoret complete int},
\ref{th: V is ideal-theoret complete int}, \ref{th: V is
irreducible} and \ref{th: V is smooth}. These results show that a
typical sequence of polynomials is regular and defines an ideal-theoretic complete
intersection which is absolutely irreducible and nonsingular.

We then apply the bounds to polynomial sequences over finite fields
to obtain numerical results, which may also be interpreted as
probabilities for sequences chosen uniformly at random. Let $\fq$ be
the finite field with $q$ elements, where $q$ is a prime power.
Multivariate polynomial systems over $\fq$ arise in connection with
many fundamental problems in cryptography, coding theory, or
combinatorics; see, e.g., \cite{WoPr05}, \cite{DiGoSc06},
\cite{CaMaPr12}, \cite{CeMaPePr14}, \cite{MaPePr14}. A random
multivariate polynomial system over $\fq$ with more equations than
variables is likely to be unsolvable over $\fq$. On the other hand,
when there are more variables than equations, the system is likely
to be solvable over $\fq$ (see \cite{FuBa09} for the phase
transition between these two regimes).

Further information can be obtained if the projective variety
$V\subset\PpF^n$ defined by $f_1\klk f_s$ possesses ``nice''
geometric properties. The projective variety $V$ is the set of
common zeros of $f_1\klk f_s$ in the $n$--dimensional projective
space $\PpF^n$ over an algebraic closure $\mathbb{F}$ of $\fq$.
Indeed, if $V$ is known to be a nonsingular or an absolutely
irreducible complete intersection, then estimates on the deviation
from the expected number of points of $V$ in $\Pp^n(\fq)$ are
obtained in \cite{Deligne74}, \cite{Hooley91}, \cite{GhLa02a},
\cite{CaMaPr15}, \cite{MaPePr15}. This motivates the study of the
``frequency'' with which such geometric properties arise.

Over finite fields, the geometric theorems plus an appropriate
version of the Weil bound yield bounds on the number of such
sequences of polynomials; see Corollaries \ref{coro: number
ideal-theoret complete inters}, \ref{coro: number abs irred} and
\ref{coro: number of V smooth}. This can be interpreted as
probabilities for polynomial sequences chosen uniformly at random.
The lower bounds tend to 1 with growing field size.

For $s=1$, the variety defined by a single polynomial $f_1\in
\K[X_0\klk X_n]$ is a hypersurface, which is absolutely irreducible
if the polynomial $f_1$ is. Counting irreducible multivariate
polynomials over a finite field is a classical subject which goes
back to the works of \cite{Carlitz63}, \cite{Carlitz65} and
\cite{Cohen68}; see \cite{MuPa13}, Section 3.6, for further
references. In \cite{GaViZi13}, exact formulas on the number of
absolutely irreducible multivariate polynomials over a finite field
and easy--to--use approximations are provided. No results on the
number of sequences of polynomials $f_1\klk f_s$ over a finite field
defining an absolutely irreducible projective variety are known to
the authors.

Concerning nonsingularity over an arbitrary field $\K$, the set of
all $s$--tuples of homogeneous polynomials $f_1\klk f_s\in\K[X_0\klk
X_n]$ of degrees $d_1\klk d_s$ defining a projective variety which
fails to be nonsingular of dimension $n-s$ is called the {\em
discriminant locus}. It is well-known that the discriminant locus is
a hypersurface of the space of $s$--tuples $f_1\klk f_s$ of
homogeneous polynomials of degrees $d_1\klk d_s$; see, e.g.,
\cite{GeKaZe94} for the case of the field of complex numbers. This
hypersurface is defined by a polynomial in the coefficients of the
polynomials $f_1\klk f_s$ which is homogeneous in the coefficients
of each $f_i$. For $s=1$, a well-known result of George Boole asserts
that the discriminant locus has degree $(n+1)(d_1-1)^n$;
see \cite{Cayley45}.
On the other hand, in \cite{Benoist12} an exact formula for the
degrees of the discriminant locus is provided. The calculation is
based on a study of dual varieties of nonsingular toric varieties in
characteristic zero. Then the case of positive characteristic is
dealt with using projective duality theory. Our approach is based on
the analysis of an incidence variety with tools of classical
projective geometry. We do not obtain exact formulas, but
easy--to--use approximations for the homogeneity degrees.

The above results assume a fixed sequence of degrees. When we vary
the degrees, it is natural to keep the B\'ezout number $\delta = d_1
\cdots d_s$ constant. In Section \ref{section:varyingDegree}, we
show that ``most'' polynomial sequences define a degenerate variety,
namely, a hypersurface in some linear projective subspace. Here,
``most'' refers to the dimension of the set of all relevant
polynomial sequences for infinite $K$, and to their number in the
case of finite $K$.

Let $d_1\klk d_s\ge 1$ be given and let $f_1\klk f_s\in K[X_0\klk
X_n]$ be homogeneous polynomials of degrees $d_1\klk d_s$ with
coefficients in an arbitrary field $K$. A basic quantity associated
to $f_1\klk f_s$ is the B\'ezout number $\delta=d_1\cdots d_s$. For
example, for $K=\fq$ the cost of several algorithms for finding a
common zero with coefficients in $\fq$ of $f_1\klk f_s$ is measured
in terms of the B\'ezout number $\delta$ (see, e.g., \cite{HuWo99},
\cite{CaMa06a}, \cite{BaFaSaSp13}). In this sense, it may be
interesting to study geometric properties that can be expected from
a typical sequence $f_1\klk f_s$ of $K[X_0\klk X_n]$ for which only
the B\'ezout number $\delta$ is given. For a given degree pattern
with B\'ezout number $\delta$, the results of the first part of this
paper show that the corresponding projective variety is expected to
be a complete intersection of dimension $n-s$ and degree $\delta$.
Therefore, the situation is somewhat reminiscent of that of the Chow
variety of projective varieties of a given dimension and degree in a
given projective space.

The Chow variety of curves of $\PpoK^n$ of degree $\delta$ over an
algebraic closure $\oK$ of a field $\K$ is considered in
\cite{EiHa92}. It is shown that its largest irreducible component
consists of planar irreducible curves provided that $\delta$ is
large enough. Over a finite field, \cite{CeGaMa13} use this to
obtain estimates, close to $1$, on the probability that a uniformly
random curve defined over a finite field $\fq$ is absolutely
irreducible and planar.
The present paper shows that for a fixed B\'ezout number, a typical
sequence of polynomials with corresponding degree pattern defines an
irreducible hypersurface $V$ in some linear projective subspace of
$\PpoK^n$ (Theorem \ref{th: number abs irred hypersurfaces}). Thus
the points of $V$ span a linear space of dimension $1 + \dim V$,
which is the minimal value unless $V$ is linear. In particular, a
typical $V$ is degenerate. Here, ``typical'' refers to the dimension
of the set of polynomial sequences, fixing the B\'ezout number.
Furthermore, for a finite field we provide nearly optimal bounds on
the number of polynomial sequences that define such degenerate
varieties. This result generalizes the corresponding one of
\cite{CeGaMa13} from curves to projective varieties of arbitrary
dimension.
%
%
\section{Notions and notations}
\label{section:notations}
We collect some basic definitions and facts, 
using standard notions and notations of algebraic geometry,
which can be found in, e.g., \cite{Kunz85} or \cite{Shafarevich94}.
The reader familiar with this material may want to skip ahead to
Section \ref{section:complete}.

Let $\K$ be a field, $\oK$ an algebraic closure, and
$\PpoK^n$ the  $n$--dimensional projective space over $\oK$.
It is endowed with its Zariski topology over
$\oK$, for which a closed set is the zero locus of homogeneous
polynomials of $\oK[X_0,\ldots, X_n]$. We shall also consider the
Zariski topology of $\PpoK^n$ over $\K$, where closed sets are zero
loci of homogeneous polynomials in $\K[X_0\klk X_n]$.

A subset $V\subset \PpoK^n$ is a {\em projective $\K$--variety} if
it is the set $Z(f_1\klk f_s)$ (or $\{f_1=0\klk f_s=0\}$) of common
zeros in $\PpoK^n$ of a family $f_1,\ldots, f_s \in\K[X_0 ,\ldots,
X_n]$ of homogeneous polynomials.

A $\K$--variety $V\subset \PpoK^n$ is $\K$--{\em irreducible} if it
cannot be expressed as a finite union of proper $\K$--subvarieties
of $V$. Further, $V$ is {\em absolutely irreducible} if it is
$\oK$--irreducible as a $\oK$--variety. Any $\K$--variety $V$ can be
expressed as a non-redundant union $V=\mathcal{C}_1\cup
\cdots\cup\mathcal{C}_r$ of irreducible (absolutely irreducible)
$\K$--varieties, unique up to reordering, which are called the {\em
irreducible} ({\em absolutely irreducible}) $\K$--{\em components}
of $V$.

For a $\K$-variety $V\subset\PpoK^n$, its {\em defining ideal}
$I(V)$ is the set of polynomials of $\K[X_0,\ldots, X_n]$ vanishing
on $V$. The {\em coordinate ring} $\K[V]$ of $V$ is defined as the
quotient ring $\K[X_0,\ldots,X_n]/I(V)$. The {\em dimension} $\dim
V$ of a $\K$-variety $V$ is the length $m$ of a longest chain
$V_0\varsubsetneq V_1 \varsubsetneq\cdots \varsubsetneq V_m$ of
nonempty irreducible $\K$-varieties contained in $V$. A
$\K$--variety $V$ is called {\em equidimensional} if all
irreducible $\K$--components of $V$ are of the same dimension $m$;
then  $V$ is of {\em pure dimension} $m$.

The {\em degree} $\deg V$ of an irreducible $\K$-variety $V$ is the
maximum number of points lying in the intersection of $V$ with a
linear space $L$ of codimension $\dim V$, for which $V\cap L$ is
finite. More generally, following \cite{Heintz83} (see also
\cite{Fulton84}), if $V=\mathcal{C}_1\cup\cdots\cup \mathcal{C}_r$
is the decomposition of $V$ into irreducible $\K$--components,
then the degree of $V$ is
$$\deg V=\sum_{1 \leq i \leq r} \deg \mathcal{C}_i.$$
The following {\em B\'ezout
inequality} holds (see \cite{Heintz83}, \cite{Fulton84}, \cite{Vogel84}):
if $V$ and $W$ are $\K$--varieties, then
\begin{equation}\label{eq: Bezout}
\deg (V\cap W)\le \deg V \cdot \deg W.
\end{equation}

Let $V\subset\PpoK^n$ be a projective variety and $I(V)\subset
\oK[X_0,\ldots, X_n]$ its defining ideal. For $x\in V$, the {\em
dimension} $\dim_x V$ {\em of} $V$ {\em at} $x$ is the maximum of
the dimensions of the irreducible components of $V$ that contain
$x$. If $I(V)=(f_1,\ldots,f_s)$, a point $x\in V$ is called {\em
regular} if the rank of the Jacobian matrix $(\partial f_i/\partial
X_j)_{1\le i\le s,0\le j\le n}(x)$ of $f_1,\ldots, f_s$ with respect
to $X_0,\ldots, X_n$ at $x$ is equal to $n-\dim_xV$. Otherwise, the
point $x$ is called {\em singular}. The set of singular points of
$V$ is the {\em singular locus} $\mathrm{Sing}(V)$ of $V$. A variety
is called {\em nonsingular} if its singular locus is empty.
%
%
\subsection{Complete intersections}
If the projective $\K$--variety $V=Z(f_1\klk
f_s)$  defined by homogeneous polynomials $f_1 \klk f_s$ in $\K[X_0\klk X_n]$
is of pure dimension $n-s$, it is a {\em
set-theoretic} {\em complete intersection} (defined over $\K$).
This is equivalent to the sequence
 $(f_1 \klk f_s)$ being a {\em
regular sequence}, meaning that $f_1$ is nonzero and each $f_i$ is neither zero
nor a zero divisor in $\K[X_0\klk X_n]/(f_1\klk f_{i-1})$ for $2\le
i\le s$.  In particular, any
permutation of a regular sequence of homogeneous polynomials
is also regular.

If the ideal $(f_1\klk f_s)$ generated by $f_1\klk f_s$ is radical,
then we say that $V$ is an {\em ideal-theoretic} {\em complete
intersection}, or simply a {\em complete intersection} (defined over
$\K$). The ``radical'' property rules out repeated components and is
the appropriate notion from an algebraic point of view.
 If $V\subset\PpoK^n$ is a complete intersection defined over
$\K$, of dimension $n-s$ and degree $\delta$, and $f_1 \klk f_s$ is
a system of homogeneous generators of $I(V)$, the degrees $d_1\klk
d_s$ depend, up to permutation, only on $V$ and not on the system of
generators (see, e.g., \cite{GhLa02a}, Section 3). Arranging the
$d_i$ in such a way that $d_1\geq d_2 \geq \cdots \geq d_s$, we call
$\bfs d=(d_1\klk d_s)$ the {\em multidegree} of $V$.

According to the B\'ezout inequality (\ref{eq: Bezout}), if
$V\subset\PpoK^n$ is a complete intersection defined over $\K$ of
multidegree $\bfs d=(d_1\klk d_s)$, then $\deg V\le d_1\cdots d_s$.
Actually, a much stronger result holds, namely,
the {\em B\'ezout theorem}:
\begin{equation}\label{eq: Bezout theorem}
\deg V=d_1\cdots d_s.
\end{equation}
See, e.g., \cite{Harris92}, Theorem 18.3, or \cite{SmKaKeTr00}, \S
5.5, page 80.

In what follows we shall deal with a particular class of complete
intersections, which we now define. A $\K$--variety is {\em regular
in codimension $m$} if the singular locus $\mathrm{Sing}(V)$ of $V$
has codimension at least $m+1$ in $V$, namely if $\dim V-\dim
\mathrm{Sing}(V)\ge m+1$. A complete intersection $V$ which is
regular in codimension 1 is called {\em normal}; actually, normality
is a general notion that agrees on complete intersections with the
one we use here. A fundamental result for projective complete
intersections is the Hartshorne connectedness theorem (see, e.g.,
\cite{Kunz85}, Theorem VI.4.2), which we now state. If
$V\subset\PpoK^n$ is a set-theoretic complete intersection defined
over $\K$ and $W\subset V$ is any $\K$--subvariety of codimension at
least 2, then $V\setminus W$ is connected in the Zariski topology of
$\PpoK^n$ over $\K$. For a normal set-theoretic complete
intersection $V$ defined over $\oK$, the subvariety
$W=\mathrm{Sing}(V)\subset V$ has codimension at least 2. Then the
Hartshorne connectedness theorem asserts that $V\setminus W$ is
connected, which implies that $V$ is absolutely irreducible.

The next statement  summarizes several well-known relations among the
concepts introduced above.
\begin{fact}\label{theorem: normal complete int implies irred}
For a projective variety $V\subset\PpoK^n$, the following hold.
\begin{itemize}
  \item If $V$ is an ideal-theoretic complete intersection, then it
  is a set-theoretic complete intersection.
  \item If $V$ is a normal set-theoretic complete intersection, then it is
  absolutely irreducible.
  \item If $V$ is nonsingular, then it is normal.
\end{itemize}
\end{fact}
%
%
%
\subsection{Multiprojective space}
Let $\N=\Z_{\ge 0}$ be the set of nonnegative integers, and let
$\bfs n=(n_1\klk n_s)\in\N^s$. We define $|\bfs n|=n_1\plp n_s$ and
$\bfs n!=n_1!\cdots n_s!$. Given $\bfs \alpha,\bfs \beta\in\N^s$, we
write $\bfs \alpha\ge\bfs \beta$ whenever $\alpha_i\ge\beta_i$ holds
for $1\le i\le s$. For $\bfs d=(d_1\klk d_s)\in\N^s$, the set
$\N_{\bfs d}^{\bfs{n+1}}=\N_{d_1}^{n_1+1}\times\cdots\times
\N_{d_s}^{n_s+1}$ consists of the elements $\bfs a=(\bfs a_1\klk\bfs
a_s)\in\N^{n_1+1}\times\cdots \times\N^{n_s+1}$ with $|\bfs
a_i|=d_i$ for $1\le i\le s$.

We denote by $\PpoK^{\bfs n}$ the multiprojective space $\PpoK^{\bfs
n}= \PpoK^{n_1}\times\cdots\times \PpoK^{n_s}$ defined over $\oK$.
For $1\le i\le s$, let $X_i=\{X_{i,0}\klk X_{i,n_i}\}$ be disjoint sets of
$n_i+1$ variables and let $\bfs X=\{X_1\klk X_s\}$. A {\em
multihomogeneous} polynomial $f\in\K[\bfs X]$ of multidegree $\bfs
d=(d_1\klk d_s)$ is a polynomial which is homogeneous of degree
$d_i$ in $X_i$ for $1\le i\le s$. An ideal $I\subset\K[\bfs X]$ is
{\em multihomogeneous} if it is generated by a family of
multihomogeneous polynomials. For any such ideal, we denote by
$Z(I)\subset\PpoK^{\bfs n}$ the variety defined (over $\K$) as its
set of common zeros. In particular, a hypersurface in $\PpoK^{\bfs
n}$ defined over $\K$ is the set of zeros of a multihomogeneous
polynomial of $\K[\bfs X]$. The notions of irreducibility and
dimension of a variety in $\PpoK^{\bfs n}$ are defined as in the
projective space.
%
%
\subsubsection{Mixed degrees}
We discuss the concept of \emph{mixed degree} of a multiprojective
variety and a few of its properties, following the exposition in
\cite{DaKrSo13}. Let $V\subset\PpoK^{\bfs n}$ be an irreducible
variety defined over $\oK$ of dimension $m$ and let
$I(V)\subset\oK[\bfs X]$ be its multihomogeneous ideal. The quotient
ring $\oK[\bfs X]/I(V)$ is multigraded and its part of multidegree
$\bfs{b}\in\N^s$ is denoted by $(\oK[\bfs X]/I(V))_{\bfs{b}}$. The
{\em Hilbert--Samuel} function of $V$ is the function
$H_V:\N^s\to\N$ defined as $H_V(\bfs{b})=\dim (\oK[\bfs
X]/I(V))_{\bfs{b}}$. It turns out that there exist
$\bfs{\delta}_0\in\N^s$ and a unique polynomial $P_V\in\Q[T_1\klk
T_s]$ of degree $m$ such that $P_V(\bfs{\delta})=H_V(\bfs{\delta})$
for every $\bfs{\delta}\in\N^s$ with $\bfs{\delta}\ge
\bfs{\delta}_0$; see \cite{DaKrSo13}, Proposition 1.8. For
$\bfs{b}\in\N^s_{\le m}$, we define the {\em mixed degree of $V$ of index}
$\bfs{b}$ as the nonnegative integer
$$\deg_{\bfs{b}}(V)=\bfs{b}!\cdot \mathrm{coeff}_{\bfs{b}}(P_V).$$
This notion can be extended to equidimensional varieties and, more
generally, to equidimensional cycles (formal integer linear combinations
of subvarieties of equal dimension) by
linearity.

The Chow ring of $\PpoK^{\bfs n}$ is the graded ring
$$A^*(\PpoK^{\bfs n})=\Z[\theta_1\klk\theta_s]/(\theta_1^{n_1+1}
\klk \theta_s^{n_s+1}),$$
where each $\theta_i$ denotes the class of the inverse image of a
hyperplane of $\PpoK^{n_i}$ under the projection $\PpoK^{\bfs
n}\to\PpoK^{n_i}$. Given a variety $V\subset\PpoK^{\bfs n}$ of pure
dimension $m$, its class in the Chow ring is
$$[V]=\sum_{\bfs{b}}\deg_{\bfs{b}}(V) \theta_1^{n_1-b_1}
\cdots \theta_s^{n_s-b_s}\in A^*(\PpoK^{\bfs n}),$$
where the sum is over all $\bfs{b}\in\N^s_{\le m}$ with $\bfs{b}\le \bfs
n$. This is an homogeneous element of degree $|\bfs n|-m$. In
particular, if $\mathcal{H}\subset\PpoK^{\bfs n}$ is a hypersurface
and $f\in\oK[\bfs X]$ is a polynomial of minimal degree defining
$\mathcal{H}$, then
\begin{equation}\label{eq: definition Chow cycle of a hypersurface}
[\mathcal{H}]=\sum_{1 \leq i \leq s} \deg_{X_i}(f)\,\theta_i ;
\end{equation}
see \cite{DaKrSo13}, Proposition 1.10.

A fundamental tool for estimates of mixed degrees involving
intersections of multiprojective varieties is the following
multiprojective version of the B\'ezout theorem, called the {\em
multihomogeneous B\'ezout theorem}; see \cite{DaKrSo13}, Theorem
1.11. If $V\subset\PpoK^{\bfs n}$ is a multiprojective variety of
pure dimension $m>0$ and $f\in\oK[\bfs X]$ is a multihomogeneous
polynomial such that $V\cap Z(f)$ is of pure dimension $m-1$, then
\begin{equation}\label{eq: multihomogeneous Bezout th}
[V\cap Z(f)]=[V]\cdot [Z(f)].
\end{equation}

Finally, the following result shows that mixed degrees are monotonic
with respect to linear projections. Let $\bfs{l}=(l_1\klk
l_s)\in\N^s$ be an $s$--tuple with $\bfs{l}\le\bfs n$ and let
$\pi:\PpoK^{\bfs n}\dashrightarrow\PpoK^{\bfs{l}}$ be the linear
projection which takes the first $l_i+1$ coordinates of each
coordinate $x_i$ of each point $\bfs{x}=(x_1\klk x_s)\in\PpoK^{\bfs
r}$, namely,
$$\pi(x_{i,j} \colon 1\le i\le s,0\le j\le n_i)=
(x_{i,j} \colon 1\le i\le s,0\le j\le l_i).$$
This rational map induces the following injective $\Z$--linear map:
$$\jmath:A^*(\PpoK^{\bfs{l}})\to A^*(\PpoK^{\bfs n}),\quad
\jmath(P)=\bfs{\theta}^{\bfs n-\bfs{l}}P.$$
If $V\subset\PpoK^{\bfs n}$ is a variety of pure dimension $m$ and
$\overline{\pi(V)}$ is also of pure dimension $m$, then
\begin{equation}\label{eq: multihomogeneous mixed degree proj}
\jmath([\overline{\pi(V)}])\le [V];
\end{equation}
see \cite{DaKrSo13}, Proposition 1.16. Equivalently, $\deg_{\bfs
b}(\pi_*V)\le\deg_{\bfs b}V$ for any $\bfs b\in\N^s_{\le m}$, where
$\pi_*V=\deg (\pi|_V)\overline{\pi(V)}$ and
$\deg(\pi|_V)=[\oK(V):\oK(\overline{\pi(V)})]$.
%
%
\subsection{Varieties over a finite field $\fq$}
In the following, $\PpF^n$ is the projective
$n$--dimensional space over an algebraic closure $\mathbb{F}$ of
$\fq$, endowed with its Zariski topology.
 $\Pp^n(\fq)$ is the $n$--dimensional projective space over
$\fq$, of cardinality
\begin{equation}
\label{sizeProjective}
p_n=\#\Pp^n(\fq)=q^n+q^{n-1}+\dots+1.
\end{equation}
We denote by $V(\fq)$ the set of $\fq$--rational points of a
projective variety $V\subset\PpF^n$, namely, $V(\fq)=V\cap
\Pp^n(\fq)$.
If $V$ is
 of dimension
$m$ and degree $\delta$, we have
 \begin{equation}\label{eq: upper bound -- projective gral}
   \#V(\fq)\leq \delta\, p_m;
 \end{equation}
%
%
see \cite{GhLa02a}, Proposition 12.1, or \cite{CaMa07}, Proposition
3.1. For $\bfs n=(n_1\klk n_s)\in\N_{\geq 1}^s$, 
$\PpF^{\bfs n}= \PpF^{n_1}\times\cdots\times \PpF^{n_s}$ is the
multiprojective space  over $\F$.
Let $f\in\F[\bfs X]$ be multihomogeneous of multidegree $\bfs
d=(d_1\klk d_s)$. The following provides a highly useful upper bound
on the number of $\fq$--rational zeros of $f$ in $\Pp^{\bfs
n}(\fq)$, which generalizes (\ref{eq: upper bound -- projective
gral}) to the multiprojective setting.

For $\bfs{\varepsilon}\in\N^s$ and $\bfs n\ge\bfs{\varepsilon}$,
we use the notations $\bfs
d^{\bfs{\varepsilon}}=d_1^{\,\varepsilon_1}\cdots d_s^{\,\varepsilon_s}$ and
$p_{\bfs n-\bfs{\varepsilon}}=p_{n_1-\varepsilon_1}\cdots p_{n_s-\varepsilon_s}$.
\begin{fact}[{\cite{CaMaPr15}, Proposition 3.1}]
\label{prop: upper bound multihomogeneous} Let $f\in\F[\bfs X]$ be a
multihomogeneous polynomial of multidegree $\bfs d$ with
 $d_i\le q$ for all $i$, and let $N$ be the number of zeros of $f$ in
$\Pp^{\bfs n}(\fq)$. Then
$$N\le
\sum_{\bfs\varepsilon\in\{0,1\}^s\setminus\{\bfs{0}\}}
(-1)^{|\bfs\varepsilon|+1}\bfs d^{\, \bfs\varepsilon}p_{\bfs
n-\bfs\varepsilon}.
$$
\end{fact}
%
%
\section{Set--theoretic and ideal-theoretic\\ complete intersections}
\label{section:complete}
It is convenient to fix the following notation:
\begin{equation}
\begin{aligned}
\label{notationD}
& \text{integers } n \text{ and } s \text{ with } 0<s<n, \\
&\bfs d=(d_1\klk d_s)\in\N^s \text{ with } d_1\ge d_2\ge \cdots \ge d_s\ge 1 \text{ and } d_1 \geq 2,\\
& \delta=d_1\cdots d_s,\\
& \sigma=(d_1-1)\plp (d_s-1),\\
& D_i=\binom{d_i+n}{n}-1 \text{ for } 1\le i\le s,\\
&  \bfs D=(D_1\klk D_s) \in \N^s,\\
& |\bfs D|=D_1\plp D_s.
\end{aligned}
\end{equation}

Let $\K$ be a field. Each $s$--tuple of homogeneous polynomials
$\bfs f=(f_1\klk f_s)$ with $f_i\in\K[X]={\K}[X_0\klk X_n]$ and
$\deg f_i=d_i$ is represented by a point in the multiprojective
space $\PpoK^{\bfs D}=\PpoK^{D_1}\times\cdots\times\PpoK^{D_s}$.
More precisely, let $\bfs \lambda=(\lambda_1\klk\lambda_s)$ be a
point of $\PpoK^{\bfs D}$. We label the $D_i+1$ coordinates of each
$\lambda_i$ by the $D_i+1$ multi-indices
$\bfs\alpha\in\N_{d_i}^{n+1}$, namely,
$\lambda_i=(\lambda_{i,\bfs\alpha}\colon |\bfs\alpha|=d_i)$. Then we
associate each point $\bfs \lambda=(\lambda_1\klk\lambda_s)$ with
the $s$--tuple of polynomials $\bfs f=(f_1\klk f_s)$ defined as
$f_i=\sum_{|\bfs\alpha|=d_i}\lambda_{\bfs\alpha,i}X^{\bfs\alpha}$
for $1\le i\le s$. In the following, the symbol $\bfs f=(f_1\klk
f_s)$ shall denote either an $s$--tuple of homogeneous polynomials
of $\K[X_0\klk X_n]$ with degree pattern $(d_1\klk d_s)$ or the
corresponding point in $\PpoK^{\bfs D}$.

Let $\{F_{i,\bfs\alpha}:|\bfs\alpha|=d_i\}$ be a set of $D_i+1$
variables over $\oK$ for $1\le i\le s$. We shall consider the formal
polynomial $F_i= \sum_{|\bfs\alpha|=d_i} F_{i,\bfs\alpha}
X^{\bfs\alpha}$, which is homogeneous of degree $d_i$ in the
variables $X_0\klk X_n$. We  use the notations
$\mathrm{coeffs}(F_i)= \{F_{i,\bfs\alpha}\colon |\bfs\alpha|=d_i\}$ for
$1\le i\le s$ and $\mathrm{coeffs}(\bfs F)=\cup_{1\le i\le
s}\mathrm{coeffs}(F_i)$. The coordinate ring of $\PpoK^{\bfs D}$ is
represented by the polynomial ring $\oK[\mathrm{coeffs}(\bfs F)]$.
The obstruction polynomials $P$ to be defined are elements
of this ring, and given some polynomial sequence $\bfs f$ as above,
it is well-defined whether $P(\bfs f)=0$ or not.
%
%
\subsection{Set--theoretic complete intersections}
\label{subsec: set-theoretic complete intersections}
We first consider the set of $s$--tuples $\bfs f$ of homogeneous
polynomials of $\K[X_0\klk X_n]$ with degree pattern $(d_1\klk d_s)$
defining a set-theoretic complete intersection. For this purpose, we
introduce the following incidence variety:
\begin{equation}\label{eq: definition incidence var W}
W=\{(\bfs f,x)\in\PpoK^{\bfs D}\times\PpoK^n:\bfs f(x)=\bfs{0}\}.
\end{equation}
This incidence variety is well-known. For the sake of completeness,
we establish here its most important geometric properties.
\begin{lemma}\label{lemma: W is abs irred}
$W$ is absolutely irreducible of dimension $|\bfs D|+n-s$.
\end{lemma}
\begin{proof}
Let $\phi:W\to\PpoK^n$ be the restriction of the projection
$\PpoK^{\bfs D}\times\PpoK^n\to\PpoK^n$ to the second argument. Then
$\phi$ is a closed mapping, because it is the restriction to $W$ of
the projection $\PpoK^{\bfs D}\times\PpoK^n\to\PpoK^n$, which is a
closed mapping.

As $W$ is a closed set of a multiprojective space, it is a
projective variety. Furthermore, each fiber $\phi^{-1}(x)$ is a
linear (irreducible) variety of dimension $|\bfs D|-s>0$, and $\phi:
W\to \PpoK^n$ is surjective. Then \cite{Shafarevich94}, \S I.6.3, Theorem 8,
shows that $W$ is irreducible.

Finally, since $W$ is defined by $s$ polynomials which form a
regular sequence of $\K[\mathrm{coeffs}(\bfs F), X]$, we see that
$\dim W=|\bfs D|+n-s$.
\end{proof}

By the theorem on the dimension of fibers, a generic $\bfs f$ as
above defines a projective variety $Z(\bfs f)\subset\PpoK^n$ of
dimension $n-s$, which is necessarily a set-theoretic complete
intersection. Our next result provides quantitative information
concerning such $\bfs f$.
\begin{theorem}\label{th: V is set-theoret complete int}
In the notation (\ref{notationD}), there exists a nonzero
multihomogeneous polynomial
$P_{\mathrm{stci}}\in\K[\mathrm{coeffs}(\bfs F)]$, of degree at most
$\delta/d_i$ in each set of variables $\mathrm{coeffs}(F_i)$ for
$1\le i\le s$, with the following property: for any $\bfs
f\in\PpoK^{\bfs D}$ with $P_{\mathrm{stci}}(\bfs f)\not=0$, the
variety $Z(\bfs f)$ has dimension $n-s$. In particular, $Z(\bfs f)$
is a set-theoretic complete intersection and $\bfs f$
is a regular sequence.
\end{theorem}
\begin{proof}
Let $\bfs f=(f_1\klk f_s)$ be an arbitrary point of $\PpoK^{\bfs
D}$. Suppose that the variety $Z(\bfs f)\subset\PpoK^n$ has
dimension $\dim Z(\bfs f)>n-s$. Then $Z(\bfs f,X_s\klk X_n)$ is not
empty. It follows that the set of $s$--tuples of polynomials $\bfs
f$ defining a variety of dimension strictly greater than $n-s$ is
contained in the set of $\bfs f$ such that $Z(\bfs f,X_s\klk X_n)$
is not empty.

The multivariate resultant of formal polynomials $F_1 \klk F_s$,
\linebreak $F_{s+1}\klk F_{n+1}$ of degrees $d_1\klk d_s,1\klk 1$ is
an irreducible multihomogeneous polynomial of
$\K[\mathrm{coeffs}(F_1)\klk \mathrm{coeffs}(F_{n+1})]$ of degree
$d_1\cdots d_{i-1}d_{i+1}\cdots d_{n+1}$ in the coefficients of
$F_i$; see \cite{CoLiOS98}, Chapter 3, Theorem 3.1. In
particular, the multivariate resultant of $F_1\klk F_s,X_s\klk X_n$
is a nonzero multihomogeneous polynomial
$P_{\mathrm{stci}}\in\K[\mathrm{coeffs}(\bfs F)]$ of degree
$\delta/d_i$ in the coefficients of $F_i$ for $1\le i\le s$.

We claim that the 
multihomogeneous polynomial $P_{\mathrm{stci}}$ satisfies the
requirements of the theorem. Indeed, let $\bfs f\in \PpoK^{\bfs D}$
with $P_{\mathrm{stci}}(\bfs f)\not=0$. Then the multivariate
resultant of $\bfs f,X_s\klk X_n$ does not vanish. According to
 \cite{CoLiOS98}, Chapter 3, Theorem 2.3, the projective
variety $Z(\bfs f,X_s\klk X_n)\subset\PpoK^n$ is empty, which
implies that $Z(\bfs f)$ has dimension at most $n-s$. On the other
hand, each irreducible component of $Z(\bfs f)$ has dimension at
least $n-s$. We deduce that $Z(\bfs f)$ is of pure dimension $n-s$.
Furthermore, as $Z(\bfs f)$ is defined by $s$ homogeneous
polynomials, we conclude that it is a set-theoretic complete
intersection. This finishes the proof of the theorem.
\end{proof}
%
%
%
\subsection{Ideal--theoretic complete intersections}
\label{subsec: ideal-theoretic complete intersections}
Now we consider the set of $s$--tuples of homogeneous polynomials
$\bfs f$ as above defining a complete intersection.

For this purpose, we introduce another incidence variety:
\begin{equation}\label{eq: definition incidence var W'}
W_{\mathrm{ci}}=\{(\bfs f,x)\in\PpoK^{\bfs D}\times\PpoK^n:\bfs
f(x)=\bfs{0},J(\bfs f)(x)=0\},
\end{equation}
where $J(\bfs f)=\det(\partial f_i/\partial X_j:1\le i,j\le s)$ is
the Jacobian determinant of $\bfs f$ with respect to $X_1\klk X_s$.
\begin{lemma}\label{lemma: W' is equidim}
$W_{\mathrm{ci}}$ is of pure dimension $|\bfs D|+n-s-1$.
\end{lemma}
\begin{proof}
We have $W_{\mathrm{ci}}=W\cap \{J(\bfs f)(x)=0\}$, where $W$ is the
incidence variety of (\ref{eq: definition incidence var W}). We
claim that $J(\bfs f)(x)$ does not vanish identically on $W$.
Indeed, fix a squarefree polynomial $f_i\in\oK[T]$ of degree $d_i$
for $1\le i\le s$ and let $f_i^h\in\oK[X_0,X_i]$ be the
homogenization of $f_i(X_i)$ with homogenizing variable $X_0$.
Denote
\begin{equation}\label{eq: part system ideal-th compl int}
\bfs f_0=(f_1^h(X_0,X_1)\klk f_s^h(X_0,X_s)).
\end{equation}
Then $\{\bfs f_0\}\times Z(\bfs f_0)$ is contained in $W$ and
$J(\bfs f_0)$ does not vanish identically on $Z(\bfs f_0)$, which
shows the claim.

According to Lemma \ref{lemma: W is abs irred}, $W$ is absolutely
irreducible of dimension $|\bfs D|+n-s$. Therefore, by the claim we
see that $W_{\mathrm{ci}}=W\cap \{J(\bfs f)(x)=0\}$ is of pure
dimension  $|\bfs D|+n-s-1$.
\end{proof}

With a slight abuse of notation, denote by
$\pi:W_{\mathrm{ci}}\to\PpoK^{\bfs D}$ the projection to the first
argument. We have the following result.
\begin{lemma}\label{lemma: pi:W' to PD is dominant}
$\pi\colon W_{\mathrm{ci}}\to\PpoK^{\bfs D}$ is a dominant mapping.
\end{lemma}
\begin{proof}
Let $\bfs f_0$ be the $s$--tuple of polynomials of (\ref{eq: part
system ideal-th compl int}). Then the fiber
$\pi^{-1}(\bfs f_0)$ has dimension $n-s-1$. Let $\mathcal{C}$ be an
irreducible component of $W_{\mathrm{ci}}$ such that $\bfs f_0\in
\pi(\mathcal{C})$. It is clear that $\dim \pi(\mathcal{C})\le |\bfs
D|$, and $\dim\mathcal{C}=|\bfs D|+n-s-1$ by Lemma \ref{lemma: W' is
equidim}. The theorem on the dimension of fibers shows that
$$\dim\mathcal{C}-\dim\pi(\mathcal{C})
=|\bfs D|+n-s-1-\dim \pi(\mathcal{C})\le \dim\pi^{-1}(\bfs
f_0)=n-s-1.$$
Thus $\dim \pi(\mathcal{C})\ge|\bfs D|$ and hence $\dim
\pi(\mathcal{C})=|\bfs D|$. It follows that
$\pi(W_{\mathrm{ci}})=\PpoK^{\bfs D}$.
\end{proof}

A consequence of Lemma \ref{lemma: pi:W' to PD is dominant} is that
a generic fiber $\pi^{-1}(\bfs f)$ has dimension $n-s-1$. In
particular, for such an $\bfs f$ the variety $Z(\bfs f)$ is of pure
dimension $n-s$. Thus, $\bfs f$ is a regular sequence of $\oK[X]$
and the hypersurface defined by the Jacobian determinant $J(\bfs f)$
intersects $Z(\bfs f)$ in a subvariety of $Z(\bfs f)$ of dimension
$n-s-1$. This implies that $\bfs f$ defines a radical ideal and
$Z(\bfs f)$ is a complete intersection.

We now turn this into quantitative information on an obstruction
polynomial whose set of zeros contains all systems not defining
a complete intersection.
\begin{theorem}
\label{th: V is ideal-theoret complete int} In the notation
(\ref{notationD}), there exists a nonzero multihomogeneous
polynomial $P_{\mathrm{ci}}\in\K[\mathrm{coeffs}(\bfs F)]$ with
$$\deg_{\mathrm{coeffs}(F_i)}P_{\mathrm{ci}} \le \delta (\frac{\sigma}{d_i}+1 )\le
2\sigma\delta$$
for $1\le i\le s$
such that any $\bfs f=(f_1\klk f_s)\in\PpoK^{\bfs D}$ with
$P_{\mathrm{ci}}(\bfs f)\not=0$ satisfies the following properties:
\begin{itemize}
  \item $f_1\klk f_s$ form a regular sequence of $\oK[X_0\klk
  X_n]$,
  \item the ideal of $\oK[X_0\klk
  X_n]$ generated by $f_1\klk f_s$ is radical.,
  \item
  $Z(\bfs f)$ is an ideal-theoretic complete intersection of dimension $n-s$
  and degree $\delta$.
\end{itemize}
\end{theorem}
\begin{proof}
Let $\bfs f=(f_1\klk f_s)$ be a point of $\PpoK^{\bfs D}$. If
$Z\big(\bfs f,J(\bfs f)\big)\subset\PpoK^n$ has dimension strictly
greater than $n-s-1$, then $Z(\bfs f,J(\bfs f),X_{s+1}\klk X_n)$ is
not empty. We conclude that the set of $\bfs f$ with $\dim
Z\big(\bfs f,J(\bfs f)\big)>n-s-1$ is contained in the set of $\bfs
f$ for which $Z(\bfs f,J(\bfs f), X_{s+1}\klk X_n)$ is not empty.

Let $\bfs f$ be a point of $\PpoK^{\bfs D}$ such that
$Z(\bfs f,J(\bfs f), X_{s+1}\klk X_n)$ is not empty. Then 
the resultant of $\bfs f,J(\bfs f),X_{s+1}\klk X_n$ must vanish. The
multivariate resultant of $F_1\klk F_s,J(\bfs F),X_{s+1}\klk X_n$ is
a nonzero polynomial $P_{\mathrm{ci}}\in\K[\mathrm{coeffs}(\bfs
F)]$. Indeed, let $\bfs f_0\in\PpoK^{\bfs D}$ be the point defined
in (\ref{eq: part system ideal-th compl int}). Then it is easy to
see that $Z(\bfs f_0,J(\bfs f_0), X_{s+1}\klk X_n)$ is empty, which
implies that $P_{\mathrm{ci}}(\bfs f_0)\not=0$.

We claim that the 
multihomogeneous polynomial
$P_{\mathrm{ci}}\in\K[\mathrm{coeffs}(\bfs F)]$ satisfies the
requirements of the theorem.

In order to show this claim, let $\bfs f\in \PpoK^{\bfs D}$ with
$P_{\mathrm{ci}}(\bfs f)\not=0$. Then the variety $Z(\bfs f,J(\bfs
f),X_{s+1}\klk X_n)$ is empty, which implies that $V'=Z(\bfs
f,J(\bfs f))$ has dimension at most $n-s-1$. On the other hand, each
irreducible component of $V'$ has dimension at least $n-s-1$ by
definition. We conclude that $V'$ is of pure dimension $n-s-1$.

Furthermore, as each irreducible component of $V=Z(\bfs f)$ has
dimension at least $n-s$, we deduce that $V'=V\cap Z(J(\bfs f))$ has
codimension at least one in $V$, and $V$ is of pure dimension $n-s$.
We conclude that $f_1\klk f_s$ form a regular sequence of $\oK[X]$
and the ideal generated by the $s\times s$ minors of the Jacobian
matrix of $f_1\klk f_s$ has codimension at least 1 in $V$. Then
\cite{Eisenbud95}, Theorem 18.15, proves that $f_1\klk f_s$ generate
a radical ideal of $\oK[X]$, so that $V$ is a complete intersection.

For an upper bound on the degree of
$P_{\mathrm{ci}}$, we use \cite{CoLiOS98}, Chapter 3, Theorem
3.1, saying that the multivariate resultant of formal polynomials $F_1\klk
F_s,F_{s+1}\klk F_{n+1}$ of degrees $d_1\klk d_s,\sigma,1\klk 1$ is
a multihomogeneous element of $\K[\mathrm{coeffs}(F_1)\klk
\mathrm{coeffs}(F_{n+1})]$ of degree $d_1\cdots d_{i-1}d_{i+1}\cdots
d_s\sigma=\sigma\delta/d_i$ in the coefficients of $F_i$ for $1\le
i\le s$ and degree $\delta$ in the coefficients of $F_{s+1}$. We
deduce that $P_{\mathrm{ci}}\in\K[\mathrm{coeffs}(\bfs F)]$ has
degree $\sigma\delta/d_i+\delta$ in the variables
$\mathrm{coeffs}(F_i)$ for $1\le i\le s$.

The third property follows from the B\'ezout theorem (\ref{eq:
Bezout theorem}).
\end{proof}

%
%
\subsubsection{Complete intersections defined over $\fq$}
From the theorem, we now derive a bound over a finite field $\fq$.
The number of all $s$--tuples of homogeneous polynomials of
$\fq[X_0\klk X_n]$ with degree sequence $\bfs d$ is
\begin{equation}
\label{allSequences}
\# \Pp^{\bfs D}(\fq) = p_{\bfs D} = \prod_{1 \le i \le s} p_{D_i}.
\end{equation}
We first present a general lower bound on the number of nonzeros of a multihomogeneous
polynomial with bounded degrees.
\begin{proposition}\label{propFromDegreeToProbability}
Let $P\in\K[\mathrm{coeffs}(\bfs F)]$ be a multihomogeneous
polynomial with $\deg_{\mathrm{coeffs}(F_i)} (P) \leq e_i \leq e
\leq q$ for $1 \leq i \leq s$, and let $N$ be the number of $\bfs
f\in\Pp^{\bfs D}(\fq)$ with $P(\bfs f) \neq 0$. Then
$$
 1- \frac{se} q  \le
  \prod_{1 \leq i \leq s}
\bigl(1-\frac{e_i}{q} \bigr) \le
\frac N  {p_{\bfs D}} \le1.
$$
The leftmost inequality assumes additionally that $q \ge es/3$.
\end{proposition}
\begin{proof}
The upper bound being obvious, we prove the lower bound.
 As $q\ge  e_i$ for all $i$, Fact \ref{prop: upper bound multihomogeneous} shows that
%
$$\#\{\bfs f\in\Pp^{\bfs D}(\fq):P(\bfs f)=0\}
\le \sum_{\bfs\varepsilon\in\{0,1\}^s
\setminus\{\bfs 0\}} (-1)^{|\bfs\varepsilon|+1}\bfs e^{\,
\bfs\varepsilon}p_{\bfs D-\bfs\varepsilon},$$
where $\bfs e=(e_1, \ldots, e_s)$.
Using the inequality
$$\frac{p_{D_i}-e_i \,
p_{D_i-1}}{p_{D_i}}\ge
1-\frac{e_i}{q} \ge 1 - \frac e q$$
for $1\le i\le s$, we
conclude that
\begin{align*}
 N & = \#\{\bfs f  \in\Pp^{\bfs D}(\fq):P (\bfs f)\not=0\} \ge
p_{\bfs D} - \!\! \sum_{\bfs \varepsilon\in\{0,1\}^s \setminus\{\bfs 0\}} \!\!
(-1)^{|\bfs\varepsilon|+1}\bfs e^{\,
\bfs\varepsilon}p_{\bfs D-\bfs\varepsilon}\\
&= \sum_{\bfs\varepsilon\in\{0,1\}^s}
(-1)^{|\bfs\varepsilon|}\bfs{e}^{\,
\bfs\varepsilon}p_{\bfs D-\bfs\varepsilon}
= \prod_{1\le i \le s}\bigl(p_{D_i}-
e_i \, p_{D_i-1}\bigr) \\
& \ge p_{\bfs D} \cdot \prod_{1 \leq i \leq s}
\bigl(1-\frac{e_i}{q} \bigr)
\ge p_{\bfs D} \bigl ( 1- \frac{e} q \bigr )^s \ge p_{\bfs D} \bigl ( 1- \frac{se} q \bigr ).
\end{align*}
The last inequality assumes $q \ge es/3$, so that in the binomial expansion
of the $s$th power, each positive even term (after the first two)
is at least as large as the following negative odd one.
%
\end{proof}
An important feature is the fact that the numerator in the lower bound
depends on the geometric system parameters $s$, $e_i$, and $e$,
but not on $q$.
This will be applied in several scenarios.
We then only state the concise leftmost lower bound. The reader can easily
substitute the more precise product lower bound if required,
also allowing a slightly relaxed lower bound on $q$.

Combining Theorem \ref{th: V is ideal-theoret complete int} and
Proposition \ref{propFromDegreeToProbability},
 we obtain the following result.

\begin{corollary}\label{coro: number ideal-theoret complete inters}
In the notation (\ref{notationD}), suppose that
$q \ge 2s\delta \sigma /3$.
Let $N_{\mathrm{ci}}$ be the number
of $\bfs f\in\Pp^{\bfs D}(\fq)$
defining a complete intersection $Z(\bfs f)\subset\PpF^n$ of
dimension $n-s$ and degree $\delta=d_1\cdots d_s$. Then
$$
 1 - \frac {2s \delta \sigma} q \le
\frac {N_{\mathrm{ci}}} {p_{\bfs D}} \le1.
$$
\end{corollary}
When $q$ is large compared to $\delta$, the
lower bound is close to $1$.
The corollary can also be interpreted as bounding the probability that
a uniformly random  $\bfs f\in\Pp^{\bfs D}(\fq)$
defines a complete intersection $Z(\bfs f)\subset\PpF^n$ of
dimension $n-s$ and degree $\delta=d_1\cdots d_s$.
%
%
\section{Absolutely irreducible  and \\ smooth complete intersections}
\label{section:irreducible}
Now we return to the general framework of the previous section, that
is, we fix an arbitrary field $\K$ and consider a sequence $\bfs
f=(f_1\klk f_s)$ of $s$ homogeneous polynomials  $f_1\klk
f_s\in\K[X]=\K[X_0\klk X_n]$ with a given degree pattern $(d_1\klk
d_s)$. In the previous section we have shown that for a generic
$\bfs f$, the projective variety $Z(\bfs f)\subset\PpoK^n$ is a
complete intersection of dimension $n-s$ and degree
$\delta=d_1\cdots d_s$.

In this section we show that $Z(\bfs f)$ is absolutely irreducible
and smooth for a generic $\bfs f$, and more precisely that the $\bfs
f$ without this property are contained in a hypersurface whose
degree we control.
%
%
\subsection{Smooth complete intersections}
\label{subsec: smooth complete intersections}
First we analyze smoothness. For this purpose, we introduce a
further incidence variety. Let $\mathcal{M}_{\bfs F}=(\partial
F_i/\partial X_j:1\le i\le s,0\le j\le n)$ denote the Jacobian
matrix of the formal homogeneous polynomials $F_1 \klk F_s$ of
degrees $d_1\klk d_s$. For $s+1\le k\le n+1$, consider the $s\times
s$--submatrix of $\mathcal{M}_{\bfs F}$ consisting of the columns
numbered $1\klk s-1$ and $k-1$, and let $J_k(\bfs F,X)$ be the
corresponding determinant, namely,
\begin{equation}\label{eq: definition jacobian determinants}
\begin{aligned}
J_k(\bfs F,X)=\det\big(\partial F_i/\partial X_j \colon & 1\le i\le
s,\\ & j\in\{1,\klk s-1,k-1\}\big).
\end{aligned}
\end{equation}
We consider the incidence variety
\begin{equation}\label{eq: definition W_s}
\begin{aligned}
W_{\mathrm {nons}}=\{ & (\bfs f,x)\in\PpoK^{\bfs D}\times\PpoK^n \colon \\
& \bfs f(x)=\bfs{0}, J_k(\bfs f)(x)=0 \text{ for } s+1\le k\le n+1\},
\end{aligned}
\end{equation}
and have the following result.
\begin{lemma}\label{lemma: J_k and f_k are regular sequence}
The polynomials $J_{s+1}(\bfs F,X)\klk J_{n+1}(\bfs F,X),F_1\klk
F_s$ form a regular sequence of $\K[\mathrm{coeffs}(\bfs F),X]$.
\end{lemma}
\begin{proof}
For $s+1\le k\le n+1$, let $\bfs{\alpha}_k=(d_1-1,0\klk 1,0\klk 0)$
be the exponent of the monomial $X_0^{d_1-1}X_{k-1}$. The choice of
$\bfs{\alpha}_k$ implies that the nonzero monomial
$F_{1,\bfs{\alpha}_k}X_0^{d_1-1}$ occurs with nonzero coefficient in
the dense representation of $\partial F_1/
\partial X_{k-1}$. Furthermore, the Jacobian determinant
$J_k(\bfs F,X)$ is a primitive polynomial of
$\K[\mathrm{coeffs}(\bfs F)\setminus\{F_{1,\bfs\alpha_k}\},X]
[F_{1,\bfs\alpha_k}]$ of degree 1 in $F_{1,\bfs\alpha_k}$. In
particular, $J_k(\bfs F,X)$ is an irreducible element of
$\oK[\mathrm{coeffs}(\bfs F),X]$. On the other hand, if $l\not= k$,
then $J_l(\bfs F,X)$ has degree zero in $F_{1,\bfs{\alpha}_k}$,
since none of the entries of the matrix defining $J_l(\bfs F,X)$
includes a derivative with respect to $X_0$ or $X_{k-1}$.

Since the multiprojective variety defined by $J_1(\bfs F,X)\klk
J_{k-1}(\bfs F,X)$ is a ``cylinder'' in the direction corresponding
to $F_{1,\bfs\alpha_k}$ and $J_k(\bfs F,X)$ is an irreducible
nonconstant element of $\K[\mathrm{coeffs}(\bfs
F)\setminus\{F_{1,\bfs\alpha_k}\},X] [F_{1,\bfs\alpha_k}]$, we
conclude that $J_k(\bfs F,X)$ is not a zero divisor modulo $J_1(\bfs
F,X)\klk$ $J_{k-1}(\bfs F,X)$ for $s+1\le k\le n+1$.

Now, denote $\bfs{\gamma}_i=(d_i,0\klk 0)$ for $1\le i\le s$.
Observe that no $J_k(\bfs F,X)$ depends on any of the indeterminates
$F_{i,\bfs{\gamma}_i}$ for $1\le i\le s$, since the partial
derivatives of $F_1\klk F_s$ with respect to $X_0$ are not included
in any of the $s\times s$--submatrices of the Jacobian matrix
$\mathcal{M}_{\bfs F}$ defining the polynomials $J_k(\bfs F,X)$. We
conclude that each $F_i$ is not a zero divisor modulo $J_{s+1}(\bfs
F,X)\klk J_{n+1}(\bfs F,X),F_1\klk F_{i-1}$. This finishes the proof
of the lemma.
\end{proof}

Now we show that for a generic $s$--tuple $\bfs f$ as above, the
corresponding system defines a smooth complete intersection. We
provide estimates on the degree of a hypersurface of $\PpoK^{\bfs
D}$ containing the elements $\bfs f$ for which $Z(\bfs f)$ is not
smooth.
\begin{theorem}\label{th: V is smooth}
In the notation (\ref{notationD}), there exists a nonzero 
multihomogeneous polynomial $P_{\mathrm{nons}}\in\K[\bfs F]$ with
$$\deg_{\mathrm{coeffs}(F_i)}P_{\mathrm{nons}}\le
\sigma^{n-s}\delta\bigl(\!\frac{\sigma}{d_i}+n-s+1\!\bigr)\le
(\sigma +n)\sigma^{n-s}\delta$$
for $1\le i\le s$ and
such that for any $\bfs f\in\PpoK^{\bfs D}$ with
$P_{\mathrm{nons}}(\bfs f)\not=0$, the variety $Z(\bfs
f)\subset\PpoK^n$ is a nonsingular complete intersection of
dimension $n-s$ and degree $\delta$.
\end{theorem}
\begin{proof}
From Lemma \ref{lemma: J_k and f_k are regular sequence} we conclude
that the incidence variety $W_{\mathrm {nons}}$ is of pure dimension
$|\bfs D|-1$. Let $\pi:\PpoK^{\bfs D}\times\PpoK^n\to\PpoK^{\bfs D}$
be the projection to the first argument. Since $\pi$ is a closed
mapping, it follows that $\pi(W_{\mathrm {nons}})$ is a closed
subset of $\PpoK^{\bfs D}$ of dimension at most $|\bfs D|-1$. In
particular, there exists $\bfs f\in\PpoK^{\bfs D}$ not belonging to
$\pi(W_{\mathrm {nons}})$, which means that the equations $\{\bfs
f(x)=\bfs{0}, J_{s+1}(\bfs f)(x)=0\klk J_{n+1}(\bfs f)(x)=0\}$
define the empty set.

Let $D_{s+1}=\cdots=D_{n+1}=\binom{\sigma+n}{n}-1$, let $\bfs
D'=(D_1\klk D_{n+1})$ and $\PpoK^{\bfs D'}=\PpoK^{\bfs D} \times
\PpoK^{D_{s+1}}\times\cdots\times\PpoK^{D_{n+1}}$. Let
$$\K[\mathrm{coeffs}(\bfs F')]= \K[\mathrm{coeffs}(\bfs
F),\mathrm{coeffs}(F_{s+1})\klk \mathrm{coeffs}(F_{n+1})]$$ and let
$P\in\K[\mathrm{coeffs}(\bfs F')]$ be the multivariate resultant of
formal polynomials $F_1 \klk F_{n+1}$ of degrees $d_1\klk
d_s,\sigma\klk \sigma$. Denote by
$\mathcal{H}_{\mathrm{gennons}}\subset\PpoK^{\bfs D'}$ the
hypersurface defined by $P$. For any $\bfs f\in\PpoK^{\bfs D}$ we
have $\bfs f\in\pi(W_{\mathrm {nons}})$ if and only if the
$(n+1)$--tuple $\big(\bfs f,J_{s+1}(\bfs f)\klk J_{n+1}(\bfs
f)\big)$ belongs to $\mathcal{H}_{\mathrm{gennons}}$. Let
$\phi:\PpoK^{\bfs D}\to\PpoK^{\bfs D'}$ be the regular mapping
defined as $\phi(\bfs f)=\big(\bfs f,J_{s+1}(\bfs f)\klk
J_{n+1}(\bfs f)\big)$. Then $\pi(W_{\mathrm {nons}})$ is the
hypersurface of $\PpoK^{\bfs D}$ defined by the polynomial
$\phi^*(P)$, where $\phi^*:\K[\mathrm{coeffs}(\bfs
F')]\to\K[\mathrm{coeffs}(\bfs F)]$ is the $\K$--algebra
homomorphism defined by $\phi$.

Next we estimate the multidegree of $\pi(W_{\mathrm {nons}})$. For
this purpose, we consider the class $[W_{\mathrm {nons}}]$ of
$W_{\mathrm {nons}}$ in the Chow ring $\mathcal{A}^*(\PpoK^{\bfs
D}\times\PpoK^n)$ of $\PpoK^{\bfs D}\times\PpoK^n$. We denote by
$\theta_i$ the class of the inverse image of a hyperplane of
$\PpoK^{D_i}$ under the $i$th canonical projection $\PpoK^{\bfs
D}\times\PpoK^n\to\PpoK^{D_i}$ for $1\le i\le s$ and by $\theta_0$
the class of the inverse image of a hyperplane of $\PpoK^n$ under
the projection $\PpoK^{\bfs D}\times\PpoK^n\to\PpoK^n$ to the second
argument. By the definition (\ref{eq: definition W_s}) of
$W_{\mathrm {nons}}$ and the multihomogeneous B\'ezout theorem
(\ref{eq: multihomogeneous Bezout th}), we obtain
\begin{align*}
[W_{\mathrm {nons}}]&=\bigl( \prod_{1 \le i \le s} (d_i\theta_0+
\theta_i) \bigr) (\sigma\theta_0+
\theta_1\plp\theta_s  )^{n-s+1}\\
&=\sigma^{n-s+1}\delta\,\theta_0^{n+1}+\sigma^{n-s}\delta\sum_{1 \le
i \le s} (\frac{\sigma}{d_i} +n-s+1 )\theta_0^n\theta_i+
\mathcal{O} (\theta_0^{n-1}),
\end{align*}
where $\mathcal{O}\big(\theta_0^{n-1}\big)$ is a sum of terms of
degree at most $n-1$ in $\theta_0$.

On the other hand, by definition $[\pi(W_{\mathrm
{nons}})]=\deg_{\mathrm{coeffs}(F_1)}\!P_{\mathrm{nons}}\,\theta_1
\plp\deg_{\mathrm{coeffs}(F_s)} \!P_{\mathrm{nons}}\,\theta_s$,
where $P_{\mathrm{nons}} \in\K[\mathrm{coeffs}(\bfs F)]$ is a
polynomial of minimal degree defining $\pi(W_{\mathrm {nons}})$. Let
$\jmath:\mathcal{A}^*\big(\PpoK^{\bfs D}\big) \hookrightarrow
\mathcal{A}^*\big(\PpoK^{\bfs D}\times\PpoK^n\big)$ be the injective
$\Z$--map $Q\mapsto\theta_0^nQ$ induced by $\pi$. Then (\ref{eq:
multihomogeneous mixed degree proj}) shows that
$\jmath([\pi(W_{\mathrm {nons}})])\le [W_{\mathrm {nons}}]$, namely,
$$\jmath([\pi(W_{\mathrm {nons}})])=
\sum_{1 \le i \le s}\deg_{\mathrm{coeffs}(F_i)}\!P_{\mathrm{nons}}
\,\theta_0^n\theta_i \le [W_{\mathrm {nons}}],
$$
where the inequality is understood in a coefficient--wise sense.
This implies
\begin{equation}\label{eq: upper bound degree C}
\deg_{\mathrm{coeffs}(F_i)}\!P_{\mathrm{nons}}\le \sigma^{n-s}\delta
(\frac{\sigma}{d_i}+n-s+1 )
\end{equation}
for $1\le i\le s$. We claim that the polynomial $P_{\mathrm{nons}}$
satisfies the requirements of the theorem.

In order to show this claim, let $\bfs f=(f_1\klk f_s)\in
\PpoK^{\bfs D}$ with $P_{\mathrm{nons}}(\bfs f)\not=0$. Then $\{\bfs
f=0,J_{s+1}(\bfs f)=0\klk J_{n+1}(\bfs f)=0\}$ is the empty
projective subvariety of $\PpoK^n$. This implies that $Z(\bfs f)$
has dimension $n-s$ and $f_1\klk f_s$ form a regular sequence of
$\oK[X]$. Furthermore, \cite{Eisenbud95}, Theorem 18.15, proves that
$f_1\klk f_s$ generate a radical ideal of $\oK[X]$. In particular,
the singular locus of $Z(\bfs f)$ is contained in $\{\bfs
f=0,J_{s+1}(\bfs f)=0\klk J_{n+1}(\bfs f)=0\}$, which is an empty
variety, showing thus that $Z(\bfs f)$ is a smooth variety.
\end{proof}

Let $\bfs f\in\PpoK^{\bfs D}$ with $P_{\mathrm{nons}}(\bfs
f)\not=0$. Then $Z(\bfs f)\subset\PpoK^n$ is a nonsingular complete
intersection which, according to Theorem \ref{theorem: normal
complete int implies irred}, is absolutely irreducible. As a
consequence, the hypersurface
$\mathcal{H}_{\mathrm{nons}}=Z(P_{\mathrm{nons}})$ contains all the
$\bfs f\in\PpoK^{\bfs D}$ for which $Z(\bfs f)$ is not absolutely
irreducible. Below we describe a
hypersurface in $\PpoK^{\bfs D}$ of lower degree which contains all
these systems (Theorem \ref{th: V is irreducible}).

In \cite{Benoist12}, Theorem 1.3, it is shown that the set of $\bfs
f\in\PpoK^{\bfs D}$ for which the variety $Z(\bfs f)\subset\PpoK^n$
is not a nonsingular complete intersection of dimension $n-s$ and
degree $\delta$ is a hypersurface of $\PpoK^{\bfs D}$. Furthermore,
the author determines exactly the degrees 
of this hypersurface.
As mentioned in the introduction to this paper, this
result is achieved by combining a study of dual varieties of
nonsingular toric varieties in characteristic zero and projective
duality theory in positive characteristic. In particular, for $s=1$
the Benoist bound 
becomes the Boole bound $(n+1)(d_1-1)^n$. On the other hand, the
bound of Theorem \ref{th: V is smooth} is
$\big((n+1)d_1-1)(d_1-1)^{n-1}$ in this case, which is fairly
close to the Boole bound.
%
%
\subsubsection{Smooth complete intersections defined over $\fq$}
Next we apply Theorem \ref{th: V is smooth} in the case $\K=\fq$.
By Theorem \ref{th: V is smooth} and
Proposition \ref{propFromDegreeToProbability},
and with $p_{\bfs D}$ from (\ref{allSequences}),
 we obtain a
lower bound, close to the trivial upper bound, on the number of those
systems that define a smooth complete intersection.
\begin{corollary}\label{coro: number of V smooth}
In the notation (\ref{notationD}), assume that 
$q \ge {s(\sigma+n) \sigma^{n-s} \delta}/3$.
Let $N_{\mathrm{nons}}$
be the number of $\bfs f\in\Pp^{\bfs D}(\fq)$
for which $Z(\bfs f)\subset\PpF^n$ is a nonsingular complete
intersection of dimension $n-s$ and degree $\delta=d_1\cdots d_s$.
Then
\begin{equation}\label{eq: number of V smooth}
1 - \frac {s(\sigma+n) \sigma^{n-s} \delta} q \le
\frac {N_{\mathrm{nons}}}
{p_{\bfs D}} \le 1.
\end{equation}
\end{corollary}
%
%
\subsection{Absolutely irreducible complete intersections}
With notations as in the previous section, in this section we obtain
an estimate on the number of polynomial systems defined over an
arbitrary field $\K$ such that the corresponding projective variety
is an absolutely irreducible complete intersection. As the approach
is similar to that of Sections \ref{subsec: ideal-theoretic complete
intersections} and \ref{subsec: smooth complete intersections}, we
shall be brief.

Let $J_{s+1}(\bfs F,X)$ and $J_{s+2}(\bfs F,X)$ be the Jacobian
determinants defined in (\ref{eq: definition jacobian
determinants}). Consider the incidence variety
\begin{equation}\label{eq: definition W_n}
W_{\mathrm{irr}}=\{(\bfs f,x)\in\PpoK^{\bfs D}\times\PpoK^n:\bfs
f(x)=\bfs{0}, J_{s+1}(\bfs f)(x)=0,J_{s+2}(\bfs f)(x)=0\}.
\end{equation}
Arguing as in the proof of Lemma \ref{lemma: J_k and f_k are regular
sequence}, we obtain the following.
\begin{lemma}\label{lemma: J_s,s+1 and f_k are regular sequence}
The polynomials $J_{s+1}(\bfs F,X), J_{s+2}(\bfs F,X),F_1 \klk F_s$
form a regular sequence of $\K[\mathrm{coeffs}(\bfs F),X]$.
\end{lemma}

Our next result asserts that for a generic $s$--tuple $\bfs f$ as
above, the corresponding variety is an absolutely irreducible
complete intersection. We also provide estimates on the degree of a
hypersurface of $\PpoK^{\bfs D}$ containing the elements $\bfs f$
for which $Z(\bfs f)$ is not absolutely irreducible.
\begin{theorem}\label{th: V is irreducible} There exists a nonzero
multihomogeneous polynomial $P_{\mathrm{irr}}\in
\K[\mathrm{coeffs}(\bfs F)]$ with
$$\deg_{\mathrm{coeffs}(F_i)} P_{\mathrm{irr}} \le
\sigma\delta (\frac{\sigma}{d_i}+2 )\le 3\sigma^2\delta
$$
for $1\le i\le s$ such that for any $\bfs f\in\PpoK^{\bfs D}$ with
$P_{\mathrm{irr}}(\bfs f)\not=0$, the variety $Z(\bfs
f)\subset\PpoK^n$ is an absolutely irreducible complete intersection
of dimension $n-s$ and degree $\delta$.
\end{theorem}
\begin{proof}
Lemma \ref{lemma: J_s,s+1 and f_k are regular sequence} shows that
$W_{\mathrm{irr}}$ is of pure dimension $|\bfs D|+n-s-2$. Let
$\pi:\PpoK^{\bfs D}\times\PpoK^n\to\PpoK^{\bfs D}$ be the projection
to the first argument. Let
$W_{\mathrm{irr}}'=W_{\mathrm{irr}}\cap\{X_{s+2}=\cdots=X_n=0\}$. As
$\pi$ is a closed mapping, $\pi(W_{\mathrm{irr}}')$ is a closed
subset of $\PpoK^{\bfs D}$. Observe that $W_{\mathrm{irr}}'$ may be
seen as an incidence variety analogous to (\ref{eq: definition W_s})
associated to generic polynomials of $\K[X_0\klk X_{s+1}]$ of
degrees $d_1\klk d_s$. Therefore, by Lemma \ref{lemma: J_k and f_k
are regular sequence} we deduce that $W_{\mathrm{irr}}'$ is of pure
dimension $|\bfs D|-1$. In particular, $\pi(W_{\mathrm{irr}}')$ has
dimension at most $|\bfs D|-1$ and hence there exists $\bfs
f\in\PpoK^{\bfs D}\setminus \pi(W_{\mathrm{irr}}')$. For such an
$\bfs f$, the equations $\{\bfs f=0,J_{s+1}(\bfs f)=0,J_{s+2}(\bfs
f)=0,X_{s+2}=0,\klk X_n=0\}$ define the empty projective set.

This shows that the multivariate resultant
$P\in\K[\mathrm{coeffs}(\bfs F)]$ of formal polynomials $F_1 \klk
F_s$ of degrees $d_1\klk d_s$ and the polynomials $J_{s+1}(\bfs
F),J_{s+2}(\bfs F),X_{s+2}\klk X_n$ is nonzero. Denote by
$\mathcal{H}_{\mathrm{irr}}\subset\PpoK^{\bfs D}$ the hypersurface
defined by $P$. Observe that
$$\mathcal{H}_{\mathrm{irr}}=\pi(W_{\mathrm{irr}}')=
\pi\big(W_{\mathrm{irr}}\cap\{X_{s+2}= \cdots=X_n=0\}\big).$$
For any 
$\bfs f\in\PpoK^{\bfs D}$, if the variety $Z\big(\bfs f,J_{s+1}(\bfs
f),J_{s+2}(\bfs f)\big)\subset\PpoK^n$ has dimension strictly
greater than $n-s-2$, then the multivariate resultant of $\bfs
f,J_{s+1}(\bfs f),J_{s+2}(\bfs f),X_{s+2}\klk X_n$ vanishes, that
is, $\bfs f$ belongs to $\mathcal{H}_{\mathrm{irr}}$. We conclude
that, if $\bfs f\notin \mathcal{H}_{\mathrm{irr}}$, then $Z\big(\bfs
f,J_{s+1}(\bfs f),J_{s+2}(\bfs f)\big)$ is of pure dimension
$n-s-2$. In particular, $Z(\bfs f)$ is a normal complete
intersection, which is absolutely irreducible by Theorem
\ref{theorem: normal complete int implies irred}.

Now we estimate the multidegree of $\mathcal{H}_{\mathrm{irr}}$. For
this purpose, we consider the class $[W_{\mathrm{irr}}']$ of
$W_{\mathrm{irr}}'$ in the Chow ring $\mathcal{A}^*(\PpoK^{\bfs
D}\times\PpoK^n)$ of $\PpoK^{\bfs D}\times\PpoK^n$. Denote by
$\theta_i$ the class of the inverse image of a hyperplane of
$\PpoK^{D_i}$ under the $i$th canonical projection $\PpoK^{\bfs
D}\times\PpoK^n\to\PpoK^{D_i}$ for $1\le i\le s$ and by $\theta_0$
the class of the inverse image of a hyperplane of $\PpoK^n$ under
the projection $\PpoK^{\bfs D}\times\PpoK^n\to\PpoK^n$ to the second
argument. By the definition (\ref{eq: definition W_n}) of
$W_{\mathrm{irr}}$ and the multihomogeneous B\'ezout theorem
(\ref{eq: multihomogeneous Bezout th}), we obtain
\begin{align*}
[W_{\mathrm{irr}}']&=\bigl(\prod_{i=1}^s(d_i\theta_0+
\theta_i)\bigr)\left(\sigma\theta_0+
\theta_1\plp\theta_s\right)^2\theta_0^{n-s-1}\\
&=\sigma^2\delta\,\theta_0^{n+1}+\sigma\delta\sum_{1 \le i \le s}
 (\frac{\sigma}{d_i} +2 )\theta_0^{n}\theta_i+
\mathcal{O}(\theta_0^{n-1}),
\end{align*}
where $\mathcal{O}\big(\theta_0^{n-1}\big)$ is a sum of terms of
degree at most $n-1$ in $\theta_0$.

On the other hand, by definition
$[\mathcal{H}_{\mathrm{irr}}]=[\pi(W_{\mathrm{irr}}')]=
\deg_{\mathrm{coeffs}(F_1)}\!P_{\mathrm{irr}}\,\theta_1 \plp
\deg_{\mathrm{coeffs}(F_s)}\!P_{\mathrm{irr}}\,\theta_s$, where
$P_{\mathrm{irr}}\in\K[\mathrm{coeffs}(\bfs F)]$ is a polynomial of
minimal degree defining $\mathcal{H}_{\mathrm{irr}}$. Let
$\jmath:\mathcal{A}^*\big(\PpoK^{\bfs D}\big) \hookrightarrow
\mathcal{A}^*\big(\PpoK^{\bfs D}\times\PpoK^n\big)$ be the injective
$\Z$--map $Q\mapsto\theta_0^nQ$ induced by $\pi$. Then (\ref{eq:
multihomogeneous mixed degree proj}) shows that
$\jmath([\pi(W_{\mathrm{irr}}')])\le [W_{\mathrm{irr}}']$, 
where the inequality is understood in a coefficient--wise sense.
This implies that, for $1\le i\le s$, the following inequality
holds:
\begin{equation}\label{eq: upper bound degree H_n}
\deg_{\mathrm{coeffs}(F_i)}\!P_{\mathrm{irr}}\le \sigma\delta
\bigl(\frac{\sigma}{d_i}+2\bigr) \le 3 \sigma^2 \delta.
\end{equation}

We claim that the multihomogeneous polynomial $P_{\mathrm{irr}}$
satisfies the requirements of the theorem. Indeed, let $\bfs
f=(f_1\klk f_s)\in \PpoK^{\bfs D}$ be such that
$P_{\mathrm{irr}}(\bfs f)\not=0$. Then $\{\bfs f=0,J_{s+1}(\bfs
f)=0,J_{s+2}(\bfs f)=0\}$ is of pure dimension $n-s-2$. It follows
that $Z(\bfs f)$ has dimension $n-s$ and $f_1\klk f_s$ form a
regular sequence of $\oK[X]$. Furthermore, \cite{Eisenbud95},
Theorem 18.15, proves that $f_1\klk f_s$ generate a radical
ideal of $\oK[X]$. In particular, the singular locus of $Z(\bfs f)$
is contained in $\{\bfs f=0,J_{s+1}(\bfs f)=0,J_{s+2}(\bfs f)=0\}$,
which has dimension $s-2$, showing that $Z(\bfs f)$ is a normal
variety, and thus absolutely irreducible.
\end{proof}

The hypersurface $\mathcal{H}_{\mathrm{irr}}\subset\Pp^{\bfs D}$ of
the proof of Theorem \ref{th: V is irreducible} is defined by the
multivariate resultant $P=P^{[0\klk\,
s+1]}\in\K[\mathrm{coeffs}(\bfs F)]$ of formal polynomials $F_1 \klk
F_s$ of degrees $d_1\klk d_s$ and the polynomials $J_{s+1}(\bfs F)$,
$J_{s+2}(\bfs F)$, $X_{s+2}\klk X_n$. It is well-known that $P$ is
actually the multivariate resultant of the polynomials $F_i(X_0\klk
X_{s+1},0\klk 0)$ for $1\le i\le s$ and $J_k(\bfs F)(X_0\klk
X_{s+1},0\klk 0)$ for $s < k\le s+2$; see, e.g., \cite{CoLiOS98}, \S
3.3, Exercise 12. In particular, $P$ only depends on the
coefficients $F_{i,\bfs \alpha}$ with $\alpha_k=0$ for $s+2\le k\le
n$. By considering the sets of indices $[0\klk s,k]$ for $s < k \le
n$, one obtains multivariate resultants $P^{[0\klk s,k]}\in
\K[\mathrm{coeffs}(\bfs F)]$  whose set of common zeros in
$\PpoK^{\bfs D}$ contains all the $\bfs f$ not defining a normal
complete intersection of dimension $n-s$ and degree $\delta$.
Furthermore, it can be proved that the polynomials $P^{[0\klk
\,s,k]}$ for $s < k\le n$ form a regular sequence of
$\K[\mathrm{coeffs}(\bfs F)]$. This shows that the set of $\bfs f\in
\PpoK^{\bfs D}$ that do not define a normal complete intersection of
dimension $n-s$ and degree $\delta$ is contained in a subvariety of
$\PpoK^{\bfs D}$ of pure codimension $n-s$.
%
%
\subsubsection{Absolutely irreducible complete intersections defined over $\fq$}
Now we apply Theorem \ref{th: V is irreducible} in the case
$\K=\fq$. Combining Theorem \ref{th: V is irreducible} and
Proposition \ref{propFromDegreeToProbability}, we can
bound the number of polynomial systems as above defining absolutely
irreducible complete intersections.
\begin{corollary}\label{coro: number abs irred}
Suppose that
$q \ge {s  \sigma^2 \delta}$.
Let
$N_{\rm irr}^{\bfs d}$ be the number of $\bfs f\in\Pp^{\bfs D}(\fq)$
such that $Z(\bfs f)\subset\PpF^n$ is an absolutely irreducible
complete intersection of dimension $n-s$ and degree
$\delta=d_1\cdots d_s$. Then
\begin{equation}\label{eq: number of V abs irred}
1 - \frac{3s  \sigma^2 \delta} q \leq
\frac {N_{\mathrm{irr}}}  {p_{\bfs D}} \le 1.
\end{equation}
\end{corollary}
%
%
\section{Absolutely irreducible complete intersections
\\ of given dimension and degree}
\label{section:varyingDegree}
%
We fix
the dimension $n\ge 2$ of a projective ambient space $\PpoK^n$
 over an algebraic closure $\oK$ of a field $K$, the
codimension $1\le s<n$ and the degree $\delta>0$, and discuss
geometric properties which are satisfied by ``most'' complete
intersections with these features. We show that most complete
intersections in this sense are absolutely irreducible hypersurfaces
within some linear projective subspace. We also provide
estimates on the number of polynomial systems defined over a finite
field $\fq$ which fail to define such an absolutely irreducible
hypersurface.

More precisely, we consider the multiprojective variety $S_0$ of all
systems  $\bfs f = (f_1, \ldots, f_s)$  of homogeneous polynomials
with $1 \leq \deg f_i \leq \delta$ for all $i$. Given a degree
pattern $\bfs d=(d_1\klk d_s)\in \N^s$ with $d_1\ge d_2\ge\cdots\ge
d_s\ge 1$ and $d_1 \cdots d_s = \delta$, the systems $\bfs f$ with
degree pattern $\bfs d$ form a closed subvariety $S_{\bfs d}$ of
$S_0$. Their union $S = \bigcup_{\bfs d} S_{\bfs d}$ over all such
$\bfs d$ is the object studied in this section. We show that for
$\bfs d^{(\delta)} = (\delta, 1, \ldots, 1)$, $S_{\bfs
d^{(\delta)}}$ is the unique component of $S$ with maximal
dimension. All systems in $S_{\bfs d^{(\delta)}}$ describe a
hypersurface within a linear subspace of codimension $s-1$, which is
proper if $s \geq 2$.

A result of a similar flavor was shown by \cite{EiHa92}. They prove
that in the Chow variety of curves of degree $\delta$ in $\PpoK^n$,
most curves are planar and irreducible if $4n-8\le \delta$. Based on
this approach, \cite{CeGaMa13} provide numerical bounds for the
probability that a curve randomly chosen in the Chow variety over a
finite field is planar and irreducible. At first sight, it may look
surprising that a generic curve in this sense is planar. We show a
corresponding result for more general varieties: the dimension of
the variety of polynomial systems defining absolutely irreducible
hypersurfaces within some linear projective subspace is larger than
the dimension of systems defining other types of varieties.

The two models of varieties are different: we consider defining
systems of polynomials, while \cite{EiHa92}
and  \cite{CeGaMa13} deal with varieties
themselves. In their case of curves, they find that unions of lines
form a component of maximal dimension within the Chow variety if
$\delta<4n-8$. The corresponding unions of linear subspaces do not
turn up in our approach.
%
%
\subsection{Dimension of systems with a given B\'ezout number}
Assume that $s\ge 2$ and for any $\bfs d = (d_1, \ldots, d_s)$ with
$d_1\ge d_2\ge \cdots\ge d_s\ge 1$ and $d_1\ge 2$, let $S_{\bfs d}$
be the multiprojective variety of all homogeneous $f_1\klk f_s\in
K[X_0\klk X_n]$ with $\deg f_i = d_i$ for all $i$. The B\'ezout
number of such a system is $\delta (\bfs d) =d_1\cdots d_s$.
 According to Theorems \ref{th: V is ideal-theoret complete int},
 \ref{th: V is smooth}, and
\ref{th: V is irreducible}, the projective variety $V=Z(\bfs
f)\subset\PpoK^n$ defined by a generic $\bfs f=(f_1\klk f_s)$ is a
smooth absolutely irreducible complete intersection of dimension
$n-s$ and degree $\delta (\bfs d)$. As the degree pattern $(d_1\klk d_s)$ is not
fixed a priori, one may wonder how frequently a given pattern
arises. We shall show that the most typical pattern is that
corresponding to hypersurfaces, namely, $(b,1\klk 1)$.

For this purpose, for any degree pattern $\bfs d=(d_1\klk
d_s)\in \N^s$ as above, we abbreviate
\begin{align*}
\delta(\bfs d)&= d_1\cdots d_s,
\quad D_i(\bfs
d)=\binom{d_i+n}{n}-1\quad\textrm{for }1\le i\le s,\\
\bfs D(\bfs d)& =(D_1(\bfs d) \klk D_s(\bfs d)),
\quad |\bfs D(\bfs d)|=D_1(\bfs d)+\cdots+D_s(\bfs d).
\end{align*}
This notation is in agreement with that of \eqref{notationD}, where
the dependence on $\bfs d$ is not explicitly indicated, since we
were considering a fixed degree pattern.

We consider the hypersurface degree pattern $\bfs d^{(b)}=(b,1\klk
1)\in\N^s$ and $\bfs D^{(b)}=\bfs D(\bfs d^{(b)})$. We start with
the following result.
\begin{lemma}\label{lemma: highest dimension} We have
$|\bfs D^{(b)}| > |\bfs D(\bfs d)|$ for all $\bfs d\not=\bfs
d^{(b)}$ with $\delta(\bfs d)=b$.
\end{lemma}
\begin{proof}
An elementary calculation shows that for $a \geq 2$ we have
$$\frac{(2a+2)!}{(a+2)!} > 2\frac{(2a)!}{a!}.$$
It follows that
$$\frac{(2a+n)!}{(a+n)!} > 2\frac{(2a)!}{a!}$$
for $n \geq 2$, since the left--hand side is monotonically
increasing in $n$. Next, we have for $a \geq c\geq 2$ that
\begin{align}
\binom{a+n}{n} & \ge \binom{c+n}{n},\nonumber \\
\frac{(ac+n)!}{(ac)!} & > 2\frac{(a+n)!}{a!}. \label{eq: lenght s=2
- aux}
\end{align}
Dividing both sides by $n!$, we find that with $s=2$,
\begin{equation}\label{eq: lenght s=2}
|\bfs D(bc,1)|>|\bfs D(b,c)|.
\end{equation}
The general claim of the lemma follows by induction on $s$.
\end{proof}

Let $a\ge c\ge 2$ and let $\rho$ be a prime number dividing $c$.
From (\ref{eq: lenght s=2 - aux}) one deduces that $|\bfs
D(a\rho,c/\rho)|>|\bfs D(a,c)|$.
For an integer $b$, we set $g(b) = 0$ if $b$ is prime, and otherwise
\begin{equation}\label{eq: definition g(delta)}
g(b)=|\bfs D^{(b)}|-|\bfs D(b/\rho, \rho,1\klk 1)| = \binom{b+n}{n}
- \binom{b/\rho+n}{n} - \binom{\rho+n}{n},
\end{equation}
where $\rho$ is the smallest prime number dividing $b$. Extending
the binomial $u(\tau) = \binom{b/\tau+n}{n}$ to a real function of
the real variable $\tau$ on the interval $[2 \twodots b/2]$ via the
gamma function, $u$ is convex and assumes its maximum at one of the
two endpoints of the interval, namely, at $\tau = 2$; see
\cite{gat11a}, (3.6). It follows that
\begin{equation}\label{eq:boundOn g(delta)}
g(b) \ge
\binom{b+n}{n} - 2 \binom{b/2+n}{n} .
\end{equation}
We always have $g(b) \ge 1$, but $g(b)$ may be quite large. For example,
if $b^2 \geq 2 n^3$, then $g(b) \geq b^2/2 n^2$.
Furthermore, for $n>s>1$ and $b\ge 2$ composite, we have
$$|\bfs D(b/\rho,\rho,1\klk 1)|=\max\limits_{\delta(\bfs d)=
b,\,\bfs d\not=\bfs d^{(b)}}|\bfs D(\bfs d)|,$$
and
\begin{equation}\label{eq: bound second highest dimension}
|\bfs D^{(b)}|\ge |\bfs D(\bfs d)|+g(b)
\end{equation}
for any $\bfs d$ with $\delta(\bfs d)=b$ and $\bfs d\not=\bfs
d^{(b)}$.

Combining Lemma \ref{lemma: highest dimension} and (\ref{eq: bound
second highest dimension}), we can conclude that among all
$s$--tuples of homogeneous polynomials having a degree pattern $\bfs
d$ with $\delta(\bfs d)=b$, ``most'' of them define a hypersurface
within some linear projective subspace of $\PpoK^n$. More precisely,
we have the following result.
\begin{corollary}\label{coro: most systs define hypersurfaces}
Let $n\ge 2$, $1\le s<n$, and $b>0$. For any degree pattern $\bfs
d\not=\bfs d^{(b)}$ with $\delta(\bfs d)=b$,
$$\dim \PpoK^{\bfs D^{(b)}}\ge
\dim \PpoKDd+g(b).$$
\end{corollary}
\begin{proof}
Since $\dim \PpoKDd=|\bfs D(\bfs d)|$ for any degree pattern
$\bfs d$, the corollary follows from \eqref{eq: bound second
highest dimension}.
\end{proof}

We may strengthen the conclusions of Corollary \ref{coro: most systs
define hypersurfaces} by applying Theorem \ref{th: V is
irreducible}.
\begin{corollary}\label{coro: most systs define
hypersurfaces - dimension} With assumptions as in Corollary
\ref{coro: most systs define hypersurfaces}, denote by
$\mathcal{S}_{\mathrm{irr}}^{\bfs d}$ the set of $\bfs f\in\PpoKDd$
with degree pattern $\bfs d$ such that $Z(\bfs f)\subset\PpoK^n$ is
an absolutely irreducible complete intersection of dimension $n-s$
and degree $b$. Then for any degree pattern $\bfs d\not=\bfs
d^{(b)}$ with $\delta(\bfs d)=b$, we have
$$\dim \mathcal{S}_{\mathrm{irr}}^{\bfs d^{(b)}}\ge
\dim \mathcal{S}_{\mathrm{irr}}^{\bfs d}+g(b).$$
\end{corollary}
\begin{proof}
Let $\bfs d$ be a degree pattern with $\delta(\bfs d)=b$. According
to Theorem \ref{th: V is irreducible}, there exists a hypersurface
$\mathcal{H}_{\mathrm{irr}}^{\bfs d}\subset\PpoKDd$ such that
$Z(\bfs f)\subset\PpoK^n$ is an absolutely irreducible complete
intersection of dimension $n-s$ and degree $b$ for any $\bfs f\in
\PpoKDd\setminus \mathcal{H}_{\mathrm{irr}}$. This implies that
$$\dim \mathcal{S}_{\mathrm{irr}}^{\bfs d}=\dim \PpoKDd=|\bfs D(\bfs d)|.$$
The conclusion now follows from Corollary \ref{coro: most systs
define hypersurfaces}.
\end{proof}

%
%
\subsection{Systems defined over a finite field}
In this section we obtain a quantitative version of Corollary
\ref{coro: most systs define hypersurfaces - dimension} for the set
of $s$--tuples of homogeneous polynomials with coefficients in $\fq$
having any degree pattern $\bfs d$ with $\delta(\bfs d)=b$. For the
case $s=1$, \cite{GaViZi13}, Corollary 6.8, shows that the number
$N_{\mathrm{irr}}^1$ of homogeneous polynomials $f_1\in \fq[X_0\klk
X_n]$ of degree $b$ which are absolutely irreducible satisfies the
following estimate:
$$
\biggl|N_{\mathrm{irr}}^1-\frac{q^{\binom{b+n}{n}}-
q^{\binom{b+n-1}{n}}}{q-1}\biggr|\le 4\,
q^{\binom{b+n-1}{n}+n-1}\frac{1- q^{-n}}{(1-q^{-1})^2},$$
where the 4 can be replaced by 3 for $n\ge 3$.

We denote as
$M_s(b)$
 the number of $\bfs d$ as in (\ref{notationD})
with $\delta(\bfs d) = b$, which equals the number of nontrivial unordered
factorizations of $b$ with at most $s$ factors, and first estimate this quantity.
 \begin{lemma}
 \label{numberFactorizations}
For positive integers $b \geq 2$ and $s$, we have $ M_s(b) \leq b^{\log_2 \log_2 b}$.
 \end{lemma}
\begin{proof}
We consider \emph{unordered factorizations} $F$ of $b \in \N$
with $s$ factors. Such an $F$ is a multiset of $s$ positive integers whose product
(with multiplicities) equals $b$. The number $1$ is allowed as a factor.
Formally, we have $F\colon \N_{\geq 1} \rightarrow \N$ with $\prod_{a \in \N_{\geq 1}} a^{F(a)} = b$
and $\sum_{a \in \N_{\geq 1}} F(a) = s$. Then $a$ ``occurs $F(a)$ times" in $F$,
and $a$ ``occurs'' in $F$ if $F(a) \geq 1$.

Picking primes $p$ and $q$ with $p \mid b$ and $q \nmid b$, we take
for any $F$ some $a$ occurring in $F$ with $p \mid a$ and replace
one copy of $a$ by $aq/p$. This new multiset $F'$ is a factorization
of $bq/p$. Replacing the unique occurrence of a multiple of $q$ in
any factorization of $bq/p$ by the same multiple of $p$ gives a
factorization of $b$. When applied to $F'$, it yields the original
$F$. Thus the map $F \mapsto F'$ is injective and $ M_s(b) \leq
M_s(bq/p)$.

Let $m = \Omega(b)$ be the number of prime factors of $b$, counted
with multiplicities, and $c$ any squarefree integer with
$m=\Omega(c)$ prime factors. The above shows that  $ M_s(b) \leq
M_s(c)$. A factorization of $c$ corresponds to a partition of $\{ 1,
\ldots, m\}$ into $t \leq s$ disjoint nonempty subsets, together
with $s-t$ times the empty set (meaning $F(1) = s-t$ in the above
notation). We drop the restriction $t \leq s$ and consider all
partitions of $\{ 1, \ldots, m\}$ into nonempty subsets. The number
of such partitions is the $m$th Bell number $B_m$. Since $M_s(2)=1$,
we may assume that $m>2$.
By \cite{bertas10},
we have
$$
\log_2 B_m \leq m \cdot \log_2 (0.8 m / \ln m) < m \log_2 m.
$$
Since $m = \Omega(b)$, we have $2^m \leq b$. It follows that
$$
M_s(b) \leq M_s(c) \leq B_m < 2^{m \log_2 m} \leq b^{\log_2 \log_2 b}.
$$
\end{proof}

Combining Lemma \ref{lemma: highest dimension} and (\ref{eq: bound
second highest dimension}) we obtain an estimate on the number of
polynomials systems as above defining a complete intersection
which is a hypersurface in some linear subspace.
As a special case of (\ref{allSequences}), the number of
$s$--tuples $\bfs f=(f_1\klk f_s)$ of homogeneous polynomials of
$\fq[X_0\klk X_n]$ with degree pattern $(b,1,\dots,1)$, up to
multiples in $\fq$ of any $f_i$, is equal to
$$\#\Pp^{\bfs D^{(b)}}(\fq)=\#\Pp^{D_b}(\fq)\cdot\big(\#
\Pp^n(\fq)\big)^{s-1}= p_{D_b}p_n^{s-1},$$
where $D_b =\binom{b+n}{n}-1$. The estimates in the following will be
expressed as a deviation from this value.
For any $\bfs d\not= \bfs d^{(b)}$ with $\delta(\bfs d)=b$,
(\ref{eq: bound second highest dimension}) implies that
\begin{equation}
\label{hypersurfacePolys}
p_{\bfs D(\bfs d)} \le \frac {p_{D_b} p_n^{s-1}} {q^{g(b)}}.
\end{equation}
\begin{theorem}\label{th: number of hypersurfaces} Let $N_{\mathrm{hyp}}$
denote the number of $\bfs f\in\Pp^{\bfs D(\bfs d)}(\fq)$ defining a
complete intersection $Z(\bfs f)\subset\PpF^n$ of dimension $n-s$
and degree $b$, which is a hypersurface in some
linear projective subspace of $\PpF^n$ for some $\bfs d$ as in (\ref{notationD})
with $\delta(\bfs d)=b$.
Then
$$\Bigl | \frac {N_{\mathrm{hyp}}} {p_{D_b}p_n^{s-1}} -1 \Bigr | \le \frac{1+9q^{-1}}{q^{n-s+3}}+
\frac{M_s(b)}{q^{g(b)}}
 \le \frac{1+9q^{-1}}{q^{n-s+3}}+
\frac{ b^{\log_2 \log_2 b}}{q^{g(b)}} .$$
\end{theorem}
\begin{proof}
Let $\bfs d=(d_1\klk d_s)\in\N^s$ be a degree pattern with $d_1\ge
d_2\ge\cdots \ge d_s\ge 1$, $\delta(\bfs d)=b$ and $\bfs d\not=\bfs
d^{(b)}$. Denote by $N_{\mathrm{ci}}^{\bfs d}$ the number of $\bfs
f\in\Pp^{\bfs D(\bfs d)}(\fq)$ defining a complete intersection of
$\PpF^n$ of dimension $n-s$ and degree $b=\delta(\bfs b)$. We have
the obvious upper bound
$$N_{\mathrm{ci}}^{\bfs d}\le p_{\bfs D(\bfs d)}=\prod_{1 \leq i \leq s} p_{D_i(\bfs d)}.$$
Any complete intersection with degree pattern $\bfs d^{(b)}$ is a
hypersurface in some linear projective subspace. Therefore,
$$\left|N_{\mathrm{hyp}}-p_{D_b}p_n^{s-1}\right|\le \big|N_{\mathrm{ci}}^{\bfs
d^{(b)}}-p_{D_b}p_n^{s-1}\big|+ \mathop{\sum_{\delta(\bfs
d)=b}}_{\bfs d\not= \bfs d^{(b)}}p_{\bfs D(\bfs d)}.
$$

The sum leads to the second summands in the bounds of the theorem
via (\ref{hypersurfacePolys}). We now consider the first term on the
right--hand side of the inequality. Let $\bfs f=(f_1\klk f_s)$ be an
$s$--tuple of polynomials of $\fq[X_0\klk X_n]$ with degree pattern
$\bfs d^{(b)}$. Without loss of generality we may assume that
$f_2=X_0\klk f_s=X_{s-2}$. Then $Z(\bfs f)\subset\PpF^n$ is a
complete intersection of dimension $n-s$ and degree $b$ if and only
if $g_1=f_1(0\klk 0,X_{s-1}\klk X_n)$ is a squarefree polynomial of
$\fq[X_{s-1}\klk X_n]$.

We fix a monomial order of $\fq[X_{s-1}\klk X_n]$ and normalize
$g_1$ by requiring that its leading coefficient with respect to this
order be equal to 1. Denote by $S_{n-s+1,b}(\fq)$ the set of
normalized squarefree homogeneous polynomials of $\fq[X_{s-1}\klk
X_n]$ of degree $b$. Then
\begin{align*}
\#S_{n-s+1,b}(\fq)\cdot q^{\binom{b+n}{n}-\binom{b+n-s+1}{n-s+1}}
=\#\{Z(\bfs f)\subset\PpF^n\mbox{ complete intersections}\
\\\mbox{ of degree }b \colon f_2=X_0\klk f_s=X_{s-2}\}.
\end{align*}

We make the previous argument for any $(f_2\klk
f_s)\in(\PpF^n(\fq))^{s-1}$ with $f_2\klk f_s$ linearly independent,
that is, for any $(s-1)$--tuple $(f_2\klk f_s)$ of homogeneous
polynomials of $\fq[X_0\klk X_n]$ with degree pattern $(1\klk 1)$
such that $f_2\klk f_s$ are linearly independent, up to multiples in
$\fq$ of any $f_i$. If $N_{\mathrm{ind}}$ is the number of elements
$(f_2\klk f_s)\in(\PpF^n(\fq))^{s-1}$ with $f_2\klk f_s$ linearly
independent, then
\begin{align}
\label{linInd} N_{\mathrm{ci}}^{\bfs d^{(b)}} & =
\#S_{n-s+1,b}(\fq)\cdot
q^{\binom{b+n}{n}-\binom{b+n-s+1}{n-s+1}} \cdot N_{\mathrm{ind}},\\
N_{\mathrm{ind}}&=\prod_{0 \leq k \leq s-2}\frac{q^{n+1}-q^k}{q-1}=p_n^{s-1}
\prod_{1 \leq k \leq s-2}\frac{q^{n+1}-q^k}{q^{n+1}-1} \le p_n^{s-1}.
\end{align}

According to \cite{GaViZi13}, Corollary 5.7, we have
$$\left|\#S_{n-s+1,b}(\fq)-
\frac{q^{\binom{b+n-s+1}{n-s+1}}}{q-1}\right|\le
\frac{3\,q^{\binom{b+n-s-1}{n-s+1}+n-s}}{(1-q^{-1})^2}.$$
Therefore, we find that
\begin{align*}
\bigg|N_{\mathrm{ci}}^{\bfs d^{(b)}}-
\frac{q^{\binom{b+n}{n}}}{q-1}N_{\mathrm{ind}}\bigg|&\le
\frac{3\,q^{\binom{b+n}{n}-\binom{b+n-s+1}{n-s+1}
+\binom{b+n-s-1}{n-s+1}+n-s}}{(1-q^{-1})^2}N_{\mathrm{ind}}\\
&\le \frac{3\,q^{\binom{b+n}{n}-\frac{b^2+n-s}{b+n-s}
\,\binom{b+n-s}{n-s}\,+n-s}}{(1-q^{-1})^2}\,p_n^{s-1}.
\end{align*}
Observe that, for $b\ge 2$ and $n-s\ge 2$,
\begin{equation}\label{eq: ineq aux}
\frac{b^2+n-s}{b+n-s}\,\binom{b+n-s}{n-s}\,-n+s\ge n-s+5.
\end{equation}
Indeed, as the left--hand side is monotonically increasing in $b$,
it suffices to consider $b=2$. For $b=2$ and $n-s\ge 2$, an
elementary calculation shows that (\ref{eq: ineq aux}) is satisfied.
As a consequence, we obtain
$$\bigg|N_{\mathrm{ci}}^{\bfs d^{(b)}}-
\frac{q^{\binom{b+n}{n}}}{q-1}N_{\mathrm{ind}}\bigg|\le
\frac{3\,q^{\binom{b+n}{n}-n+s-5}}{(1-q^{-1})^2}p_n^{s-1} \le
\frac{13}{2}\,\frac{p_{D_b}p_n^{s-1}}{q^{n-s+4}}.$$

%
%
%

%
In order to get rid of the term $N_{\mathrm{ind}}$, we use
$$\frac{q^{n+1}-q^k}{q^{n+1}-1}\ge 1-\frac{1}{q^{n+1-k}}$$
for $1\le k\le n+1$, and thus
$$1\ge \prod_{1 \leq k \leq s-2}\frac{q^{n+1}-q^k}{q^{n+1}-1}\ge
\prod_{1 \leq k \leq s-2} \bigl(1-\frac{1}{q^{n+1-k}}\bigr)\ge
1-\frac{1+2q^{-1}}{q^{n-s+3}},$$
%
It follows that
\begin{equation}\label{eq: bounds for N_ind}
p_n^{s-1}-q^{-n+s-3}(1+2q^{-1})p_n^{s-1}\le N_{\mathrm{ind}}\le p_n^{s-1}.
\end{equation}

We deduce that
\begin{align*}
\big|N_{\mathrm{ci}}^{\bfs d^{(b)}}-p_{D_b}p_n^{s-1} \big|
&\le\bigg|\frac{q^{\binom{b+n}{n}}}{q-1}
N_{_{\mathrm{ind}}}-p_{D_b}p_n^{s-1} \bigg|+
\frac{13}{2}\,\frac{p_{D_b}p_n^{s-1}}{q^{n-s+4}}\\
&\le \frac{1+2q^{-1}}{q^{n-s+3}}p_{D_b}p_n^{s-1}+
\frac{p_n^{s-1}}{q-1}+
\frac{13}{2}\,\frac{p_{D_b}p_n^{s-1}}{q^{n-s+4}}.
\end{align*}
The statement of the theorem readily follows.
\end{proof}

The error term in Theorem \ref{th: number of hypersurfaces}
decreases with growing $q$.

Next we estimate the number of polynomial systems as above defining
an absolutely irreducible complete intersection. In view of Lemma
\ref{lemma: highest dimension} and Theorem \ref{th: number of
hypersurfaces}, we have to pay particular attention to the degree
pattern $\bfs d^{(b)}$, which is the subject of the next result.
\begin{lemma}\label{lemma: estim abs irred hypersurfaces}
Let $N_{\mathrm{irr}}^{\bfs d^{(b)}}$ be the number of $\bfs
f\in\Pp^{\bfs D^{(b)}}(\fq)$ defining an absolutely irreducible
complete intersection $Z(\bfs f)\subset\PpF^n$ of dimension $n-s$
and degree $b$  which is a hypersurface in some linear projective
subspace of $\PpF^n$. Then
$$\Bigl | \frac {N_{\mathrm{irr}}^{\bfs d^{(b)}}} {p_{D_b}p_n^{s-1}} -1 \Bigr|
\le\frac{1+14\,q^{-1}}{q^{n-s+3}}. $$
%
for $b>2$ or $n-s>3$. For $b=2$ and $n-s\le 3$, the statement holds
with $1+14\,q^{-1}$ replaced by $14\,q^2$.
\end{lemma}
\begin{proof}
Let $\bfs f=(f_1\klk f_s)$ be an $s$--tuple of homogeneous
polynomial of $\fq[X_0\klk X_n]$ with degree pattern $\bfs d^{(b)}$.
Without loss of generality, we may assume that $f_2=X_0\klk
f_s=X_{s-2}$. Thus $Z(\bfs f)\subset\PpF^n$ is absolutely
irreducible if and only if $g_1=f_1(0\klk 0,X_{s-1}\klk X_n)$ is an
absolutely irreducible polynomial of $\fq[X_{s-1}\klk X_n]$. We
normalize $g_1$ by requiring that its leading coefficient with
respect to a given monomial order of $\fq[X_{s-1}\klk X_n]$ is equal
to 1. Denote by $A_{n-s+1,b}(\fq)$ the set of normalized absolutely
irreducible polynomials of $\fq[X_{s-1}\klk X_n]$ of degree $b$. We
have
\begin{align*}
\#A_{n-s+1,b}(\fq)\,q^{\binom{b+n}{n}-\binom{b+n-s+1}{n-s+1}}
=\#\{ & Z(\bfs f)\subset\PpF^n \mbox{ absolutely
irreducible} \colon \\
&f_2=X_0\klk f_s=X_{s-2}\}.
\end{align*}
Now we let $(f_2\klk f_s)$ run through all the $(s-1)$--tuples of
homogeneous polynomials of $\fq[X_0\klk X_n]$ of degree pattern
$(1\klk 1)$ with $f_2\klk f_s$ linearly independent, up to multiples
in $\fq$ of any $f_i$. If $N_{\mathrm{ind}}$ denotes the number of
elements $(f_2\klk f_s)\in(\PpF^n(\fq))^{s-1}$ with $f_2\klk f_s$
linearly independent, then
$$N_{\mathrm{irr}}^{\bfs d^{(b)}}=\#A_{n-s+1,b}(\fq)
\,q^{\binom{b+n}{n}-\binom{b+n-s+1}{n-s+1}} N_{\mathrm{ind}}.$$

From \cite{GaViZi13}, Corollary 6.8, we have that
$$\Bigl | \#A_{n-s+1,b}(\fq)-
\frac{q^{\binom{b+n-s+1}{n-s+1}}}{q-1} \Bigr | \le
\frac{3\,q^{\binom{b+n-s}{n-s+1}+n-s}}{(1-q^{-1})^2}.$$
Further, by \eqref{eq: bounds for N_ind} we have
$|N_{\mathrm{ind}}-p_n^{s-1}|\le q^{-n+s-3}(1+2q^{-1})p_n^{s-1}$.
As a consequence,
\begin{align*}
\left|N_{\mathrm{irr}}^{\bfs d^{(b)}}-p_{D_b}p_n^{s-1}
\right|\le&\left|N_{\mathrm{irr}}^{\bfs d^{(b)}}-\#A_{n-s+1,b}(\fq)
\,q^{\binom{b+n}{n}-\binom{b+n-s+1}{n-s+1}}p_n^{s-1} \right|\\
&+ \left|\#A_{n-s+1,b}(\fq)
\,q^{\binom{b+n}{n}-\binom{b+n-s+1}{n-s+1}}p_n^{s-1}
-p_{D_b}p_n^{s-1}\right|\\
\le&\,p_{D_b}\left | N_{\mathrm{ind}}-p_n^{s-1} \right | \\&+
p_{D_b}p_n^{s-1}\frac{q-1}{q^{\binom{b+n-s+1}{n-s+1}}} \Bigl | \#A_{n-s+1,b}(\fq)
-\frac{q^{\binom{b+n-s+1}{n-s+1}}}{q-1} \Bigr | \\
\le&\Big(\frac{1+2q^{-1}}{q^{n-s+3}}+
12\,q^{-\binom{b+n-s}{n-s}+n-s+1}\Big) \, p_{D_b}p_n^{s-1}.
\end{align*}
Finally, taking into account that
$$-\binom{b+n-s}{n-s}+n-s+1\le -n+s-4$$
for $b>2$ or $n-s>3$, the statement of the lemma readily follows.
\end{proof}

Now we are ready to estimate the number of polynomials systems as
above defining an absolutely irreducible projective subvariety of
$\PpF^n$ of dimension $n-s$ and degree $b$ defined over $\fq$.
We recall $g(b)$ from (\ref{eq: definition g(delta)}).
\begin{theorem}\label{th: number abs irred hypersurfaces}
Let $N_{\mathrm{irr}}^b$ be the number of $\bfs f\in\Pp^{\bfs D(\bfs
d)}(\fq)$ defining an absolutely irreducible complete intersection
$Z(\bfs f)\subset\PpF^n$ of dimension $n-s$ and degree
$b$  which is a hypersurface in some
linear projective subspace of $\PpF^n$,
for any $\bfs d$ with $\delta(\bfs d)=b$.
Then
$$\Bigl | \frac {N_{\mathrm{irr}}^b} {p_{D_b}p_n^{s-1}} -1 \Bigr|\le
\frac{1+14\,q^{-1}}{q^{n-s+3}}+
\frac{b^{\log_2 \log_2 b}}{q^{g(b)}} .$$
%
if $b>2$ or $n-s>3$.  For $b=2$
and $n-s\le 3$, the statement holds with $1+14\,q^{-1}$ replaced by
$14\,q^2$.
\end{theorem}
\begin{proof}
%
Let $N_{\mathrm{irr}}^{\not=\bfs d^{(b)}}$ denote the number of
$\bfs f\in\Pp^{\bfs D(\bfs d)}(\fq)$ such that $Z(\bfs f)$ is an
absolutely irreducible complete intersection of dimension $n-s$ and
degree $b$, not having degree pattern $\bfs d^{(b)}$. We have
$$\left|N_{\mathrm{irr}}^b-p_{D_b}p_n^{s-1}\right|\le
\big|N_{\mathrm{irr}}^{\bfs d^{(b)}}-p_{D_b}p_n^{s-1}\big|+
N_{\mathrm{irr}}^{\not=\bfs d^{(b)}}.$$

On the one hand, Lemma \ref{lemma: estim abs irred hypersurfaces}
provides an upper bound for the first term in the right--hand side.
On the other hand, by Lemmas \ref{lemma: highest dimension} and
\ref{numberFactorizations} and
(\ref{eq: bound second highest dimension}), we find
\begin{equation}\label{eq: aux th abs irred varying degree pattern}
\begin{aligned}
N_{\mathrm{irr}}^{\not=\bfs d^{(b)}} & \le \mathop{\sum_{\delta(\bfs
d)=b}}_{\bfs d\not= \bfs d^{(b)}}p_{\bfs D(\bfs d)}\le
\mathop{\sum_{\delta(\bfs d)=b}}_{\bfs d\not= \bfs
d^{(b)}}\frac{p_{D_b}
p_n^{s-1}}{q^{g(b)}} \\
& \le M_s(b) \,\frac{p_{D_b}
p_n^{s-1}}{q^{g(b)}} \le
 b^{\log_2 \log_2 b} \,\frac{p_{D_b}
p_n^{s-1}}{q^{g(b)}}.
\end{aligned}
\end{equation}
Combining both inequalities, the theorem follows.
\end{proof}

We may express Theorem \ref{th: number abs irred hypersurfaces} in
terms of probabilities. Consider the set of all $\bfs f\in\Pp^{\bfs
D(\bfs d)}(\fq)$ when $\bfs d$ runs through all the degree patterns
with $\delta(\bfs d)=b$. If $\mathcal{P}_{\mathrm{irr}}^b$ denotes
the probability for a uniformly random $\bfs f$ to define an absolutely
irreducible complete intersection $Z(\bfs f)\subset\PpF^n$ of
dimension $n-s$ and degree $b$, then Theorem \ref{th: number abs
irred hypersurfaces} and \eqref{eq: aux th abs irred varying degree
pattern} say that
$$\mathcal{P}_{\mathrm{irr}}^b\ge
1-\frac{1+14\,q^{-1}}{q^{n-s+3}}
-\frac{2 b^{\log_2 \log_2 b}}{q^{g(b)}}$$
for $b>2$ or $n-s>3$.

\section{Open questions}
Several issues are left open in the context of this work.
\begin{itemize}
\item
We have worked exclusively in the projective setting and it remains
to adapt our approach to the affine case.
\item
The nonvanishing of our obstruction polynomials is sufficient to guarantee
the property that they work for.
Can one find exact obstructions for our properties that are necessary
and sufficient? We have not even determined the dimensions of
the sets of systems that violate the property.
\item
For a particular case of the previous question, see the remarks
after \eqref{eq: bound second highest dimension}. In that context,
can one determine the dimension of the set of $\bfs f\in \PpoK^{\bfs
D}$ not defining a normal, or absolutely irreducible, complete
intersection of dimension $n-s$ and degree $\delta$? Do both
dimensions agree? Are they equidimensional subvarieties of
$\PpoK^{\bfs D}$?
\item
Elucidate the relation between the two models of varieties:
systems of defining equations as in this paper, and Chow varieties.
For example, unions of lines occur in the Chow point of view for
curves in higher-dimensional spaces, but not in our considerations.
More specifically: what is the dimension of the set of
 systems of $s$ polynomials that define
finite unions of linear spaces, each of codimension $s$?
By \cite{kum10}, such a union is a set-theoretic complete intersection
if and only if it is connected (in the Zariski topology).
\item
Stephen Watt pointed out that one might investigate the genericity
of computational ``niceness'' properties,
such as a Gr\"obner basis computation in singly-exponential time.
\end{itemize}
%
%

\end{document}

\end{document}

In particular, the polynomial
\begin{equation}\label{eq: aux pol number hypersurfaces}
g_k(X_{s-1},X_k)=f_1(0\klk 0,X_{s-1},0\klk 0,X_k,0\klk 0)
\end{equation}
is not squarefree for $s\le k\le n$.

Now we determine when the polynomials $g_k$ of (\ref{eq: aux pol
number hypersurfaces}) corresponding to a specialization of a
generic polynomial $\bfs\Lambda_1\cdot X$ of degree $b$ are not
squarefree for $s\le k\le n$. If this happens, then the
discriminant $\Delta_k\in\fq[\bfs\Lambda_1]$ with respect to $X_k$
of the polynomial
$$G_k(\bfs\Lambda_1,X_k)=
\bfs\Lambda_1\cdot (0\klk 0,1,0\klk 0,X_k,0\klk 0)$$
must vanish for $s\le k\le n$. By the expression of the discriminant
of a generic univariate polynomial over a given field (see, e.g.,
\cite{FrSm84}), we easily conclude that the polynomials
$\Delta_s\klk\Delta_n$ form a regular sequence of
$\fq[\bfs\Lambda_1]$. This shows that the set of elements
$\bfs\lambda_1\in \PpF^{D_b-1}$ for which the polynomials $g_k$ of
(\ref{eq: aux pol number hypersurfaces}) are not squarefree for
$s\le k\le n$ is contained in a projective variety of codimension
$n-s+1$ and degree at most $\big(b(b-1)\big)^{n-s+1}$. By (\ref{eq:
upper bound -- projective gral}) we deduce the bound
\begin{align*}
\left|N_{ci}^{\bfs d_b}-p_{D_b}p_n^{s-1}\right|&\le
\big(b(b-1)\big)^{n-s+1}p_{D_b-n+s-1}p_n^{s-1}\\ &\le
\bigg(\frac{b(b-1)}{q}\bigg)^{n-s+1}p_{D_b}p_n^{s-1}.
\end{align*}

On the other hand, by (\ref{eq: bound second highest dimension}) we
easily conclude that
${p_{\bfs D}}/{p_{D_b}p_n^{s-1}}\le q^{-g(d)}$
for any degree pattern $\bfs d\not=\bfs d_b$.